\documentclass[11pt]{amsart}
\usepackage{amsmath,amssymb,amsthm,mathtools}
\usepackage[margin=25mm]{geometry}
\usepackage{enumitem}
\usepackage{tikz-cd}
\usepackage[hidelinks]{hyperref}

\newcommand{\R}{\mathbb{R}}
\newcommand{\Z}{\mathbb{Z}}
\newcommand{\Q}{\mathbb{Q}}
\newcommand{\Tor}{\mathbb{T}}
\newcommand{\Dcat}{\mathsf{Diff}}
\newcommand{\Elast}{\mathsf{Elast}}
\newcommand{\Elastc}{\mathsf{Elast}_{\mathrm{curve}}}
\newcommand{\Fam}{\mathcal{F}}
\newcommand{\VF}{\mathfrak{X}}

\DeclareMathOperator{\id}{id}
\DeclareMathOperator{\pr}{pr}
\DeclareMathOperator{\dom}{dom}
\DeclareMathOperator{\divi}{div}

\DeclareMathOperator{\Path}{Path}

\theoremstyle{plain}
\newtheorem{thm}{Theorem}[section]
\newtheorem{prop}[thm]{Proposition}
\newtheorem{lem}[thm]{Lemma}
\newtheorem{cor}[thm]{Corollary}
\newtheorem{mainthm}{Theorem}

\theoremstyle{definition}
\newtheorem{defn}[thm]{Definition}
\newtheorem{ex}[thm]{Example}
\newtheorem{rem}[thm]{Remark}
\newtheorem{conv}[thm]{Convention}

\newtheorem{fact}[thm]{Fact}

\numberwithin{equation}{section}

\title{Riemannian Structures on Quotients in Elastic Diffeology via Q-Charts}
\author[Y. Shiobara]{Yusuke Shiobara}
\address{%
  Department of Science and Technology,
  Graduate School of Medicine, Science and Technology,
  Shinshu University,
  Matsumoto, Nagano 390-8621, Japan}
\email{24hs603d@shinshu-u.ac.jp}

\begin{document}

\begin{abstract}
Following Miyamoto, we consider Q-charts, a class of subductions $\phi\colon X\to Y$ of diffeological spaces along which Blohmann's elasticity descends.
We show that the Riemannian apparatus of elastic diffeology---vertical, horizontal and full connections, Riemannian metrics, Levi-Civita connections and geodesics---also descends along Q-charts.
More precisely, structures on $Y$ correspond bijectively to structures on $X$ that are invariant under the pseudogroup $\Psi(\phi)$ of fibre-preserving transitions, and torsion-freeness, effectivity, flatness, fullness, the Levi-Civita property and the geodesic equation are preserved in both directions.
The key input is that the iterated tangent bundles $T^nX$ and their fibre products $T_kX$ are the pullbacks along $\phi$ of the corresponding bundles over $Y$.

For quotients by principal actions of diffeologically discrete groups, we lift geodesics globally, prove that the class of elastic spaces on which $\R$ is a curve object is closed under such quotients, and show that complete geodesic flows descend.
As an application, we determine the Riemannian geometry of the irrational tori $\Tor_\alpha=\R/(\Z+\alpha\Z)$ completely.
Let $\vartheta$ be the one-form on $\Tor_\alpha$ induced by $dt$ on $\R$.
Every Riemannian metric is a positive constant multiple of $\vartheta^2$, and the vertical connections form a one-parameter family of flat effective connections, a single member of which is the unique Levi-Civita connection of every metric.
This connection has a complete geodesic flow, yet the induced pseudo-distance vanishes identically.
\end{abstract}

\maketitle

% =====================================================================
\section{Introduction}\label{sec:intro}
% =====================================================================
Diffeology, introduced by Souriau and systematically developed by Iglesias-Zemmour~\cite{PIZ13}, generalizes smooth manifolds: the category $\Dcat$ of diffeological spaces contains the manifolds as a full subcategory and is complete, cocomplete and cartesian closed, so that singular quotients and many infinite-dimensional spaces are treated on the same footing.
While differential forms, de~Rham theory and homotopy theory \cite{PIZ13,Kihara} extend well to this setting, Riemannian geometry is less developed, partly because the various notions of tangent space on diffeological spaces do not agree in general~\cite{Taho}.
Satisfactory notions of Riemannian metric are available through the work of Iglesias-Zemmour~\cite{PIZ25} and its interpretation via Blohmann's tangent functor in~\cite{KSS}.
Connections, geodesics and variational theory, however, require a tangent structure in the sense of Cockett and Cruttwell~\cite{CC17}, and Blohmann introduced elastic diffeological spaces precisely so that his tangent functor $T$, the left Kan extension of the tangent functor of Euclidean domains, carries such a structure on the category $\Elast$ of elastic spaces~\cite{Blo}.

Building on this structure, the author developed in~\cite{Shio} Riemannian geometry on elastic spaces: connections, Riemannian metrics, Levi-Civita connections, curvature, geodesics and the first variation formula, with mapping spaces as the principal class of examples.
Quotient spaces form another basic class of examples in diffeology, and they are the subject of this paper.

Quotients are delicate for the tangent functor.
The orbifold $\R/\{\pm1\}$ is not elastic \cite[Ex.~3.11]{Miy}, so no tangent structure, let alone a connection, is available on it.
Miyamoto \cite{Miy} isolated a class of quotient maps, the \emph{Q-charts} (Definition~\ref{def:qchart}; the notion goes back to Barre, see \cite[Rem.~3.15]{Miy}), along which elasticity descends \cite[Thm.~3.24]{Miy}; examples include quotients by free actions of countable groups and quotients $G\to G/N$ of diffeological groups by diffeologically discrete subgroups \cite[Props.~3.16, 3.17]{Miy}, in particular the irrational tori $\Tor_\alpha=\R/(\Z+\alpha\Z)$, $\alpha\notin\Q$.
Our first result says that the Riemannian structures of \cite{Shio} descend along Q-charts, and that the descent is a bijection onto the structures on the quotient.

For a Q-chart $\phi\colon X\to Y$ let $\Psi(\phi)$ be the pseudogroup of transitions $h$ of $X$ with $\phi\circ h=\phi$ (Definition~\ref{def:transition}).
A structure on $X$ (a morphism between iterated tangent bundles of $X$) is \emph{$\Psi(\phi)$-invariant} if it commutes with the tangent maps of all $h\in\Psi(\phi)$ (Definition~\ref{def:psi-invariant}).

\begin{mainthm}[Theorem~\ref{thm:transport}]\label{mainthm:A}
Let $\phi\colon X\to Y$ be a Q-chart with $X$ elastic.
For every $\Psi(\phi)$-invariant vertical connection $K$, horizontal connection $H$ and Riemannian metric $g$ on $X$ there exist a unique vertical connection $K_Y$, a unique horizontal connection $H_Y$ and a unique Riemannian metric $g_Y$ on $Y$ such that
\[
  K_Y\circ T^2\phi=T\phi\circ K,\qquad
  H_Y\circ T_2\phi=T^2\phi\circ H,\qquad
  g_Y\circ T_2\phi=g .
\]
Conversely, every vertical connection, horizontal connection and Riemannian metric on $Y$ arises in this way from a unique $\Psi(\phi)$-invariant one on $X$.
Moreover $K$ is torsion-free, effective or flat if and only if $K_Y$ is; $(K,H)$ is a full connection if and only if $(K_Y,H_Y)$ is; $K$ is a Levi-Civita connection of $g$ if and only if $K_Y$ is one of $g_Y$; each $T_x\phi$ is a linear isometry; and a path $\gamma$ in $X$ is a geodesic of $K$ if and only if $\phi\circ\gamma$ is a geodesic of $K_Y$.
\end{mainthm}

The proof rests on two facts.
The first is formal: subductions are the coequalizers of their kernel pairs, so $\Psi(\phi)$-invariant morphisms descend (Proposition~\ref{prop:descent}).
The second is that a Q-chart induces pullback squares at the tangent level: for $F=T^n$ and $F=T_k$ the map $(\pi,F\phi)\colon FX\to X\times_Y FY$ is a diffeomorphism (Proposition~\ref{prop:etale}), although a Q-chart need not be a local diffeomorphism (Remark~\ref{rem:not-local-diffeo}).
This allows us to lift structures from $Y$ to $X$ and to detect equalities on $X$ after applying $F\phi$.
Both directions of Theorem~\ref{mainthm:A} then reduce to the naturality of the structure maps of the tangent category.

What Theorem~\ref{mainthm:A} does not provide is the global lifting of curves, which is needed to compare solutions of differential equations on $X$ and $Y$.
It becomes available when $\phi$ is the quotient map of a principal action of a diffeologically discrete group $\Lambda$, i.e.\ a free action whose division map is smooth (Definition~\ref{def:principal-action}).
Let $\Elastc$ denote the class of elastic spaces on which $\R$ is a curve object in the sense of Cockett--Cruttwell--Lemay \cite{CCL}, the natural setting for the geodesic theory of \cite[\S6]{Shio}.

\begin{mainthm}[Section~\ref{sec:quotient-group}]\label{mainthm:B}
Let a diffeologically discrete group $\Lambda$ act principally on an elastic space $X$, and let $\phi\colon X\to Y=X/\Lambda$ be the quotient map.
Then $\phi$ is a Q-chart, and $\Psi(\phi)$-invariance means $\Lambda$-invariance.
Let $(K,H)$ be a $\Lambda$-invariant full connection and let $g$ be a $\Lambda$-invariant Riemannian metric.
\begin{enumerate}[label=\textup{(\roman*)},leftmargin=*]
\item Every geodesic of $K_Y$ lifts to a geodesic of $K$, uniquely once the initial point is fixed; the geodesics of $Y$ correspond bijectively to the $\Lambda$-orbits of geodesics of $X$.
\item If $X\in\Elastc$, then $Y\in\Elastc$.
\item If $X\in\Elastc$ and the geodesic spray of $(K,H)$ is complete, then so is that of $(K_Y,H_Y)$; the geodesic flows are intertwined by $T\phi$, and every $w\in TY$ is the initial velocity of exactly one geodesic of $Y$.
\end{enumerate}
\end{mainthm}

The irrational tori are the simplest spaces to which Theorem~\ref{mainthm:B} applies, and they are highly degenerate from the topological point of view: their D-topology is indiscrete and their only smooth functions are the constants.
Nevertheless they carry a geodesically complete Riemannian geometry in the sense of \cite{Shio}, which we determine explicitly.
Write $q\colon\R\to\Tor_\alpha$ for the quotient map and denote by $\bar\partial$ and $\vartheta$ the descended unit vector field and $1$-form, respectively, so that $Tq\circ\partial=\bar\partial\circ q$ and $q^*\vartheta=dt$.

\begin{mainthm}[Theorem~\ref{thm:irr}, Proposition~\ref{prop:d-zero}]\label{mainthm:C}
Let $\alpha\in\R\setminus\Q$.
\begin{enumerate}[label=\textup{(\roman*)},leftmargin=*]
\item $C^\infty(\Tor_\alpha,\R)=\R$, $\VF(\Tor_\alpha)=\R\,\bar\partial$ and $\Omega^1(\Tor_\alpha)=\R\,\vartheta$.
\item The Riemannian metrics on $\Tor_\alpha$ are exactly $\rho\,\vartheta^2$ with $\rho\in\R_{>0}$ a constant.
\item The vertical connections on $\Tor_\alpha$ form a one-parameter family; all of them are torsion-free, effective and flat.
\item For every constant $\rho>0$, the metric $\rho\,\vartheta^2$ has a unique Levi-Civita connection; it is independent of $\rho$, namely the connection $K_0$ induced by the trivial connection on $\R$.
  The geodesics of $K_0$ are the paths $t\mapsto[x+tv]$ $(x,v\in\R)$, and its geodesic flow is complete.
\item The pseudo-distance induced by $\rho\,\vartheta^2$ vanishes identically; nevertheless, for $[x]\neq[y]$ no path from $[x]$ to $[y]$ realizes the infimum defining it, that is, every such path has positive length.
\end{enumerate}
\end{mainthm}

Thus the Riemannian geometry of \cite{Shio} sees on $\Tor_\alpha$ a one-parameter family of metrics, a unique Levi-Civita connection and a complete geodesic flow, while the D-topology of $\Tor_\alpha$ is indiscrete and the induced pseudo-distance vanishes identically.
The latter is consistent with the distance theorem of \cite{KSS}, whose separation hypothesis fails here (Remark~\ref{rem:KSS-distance}).

\subsection*{Organization}
Section~\ref{sec:prelim} recalls the tangent structure on $\Elast$, the Riemannian notions of \cite{Shio} and Miyamoto's Q-charts, and proves the two basic facts mentioned above.
Section~\ref{sec:transport} proves Theorem~\ref{mainthm:A}.
Section~\ref{sec:quotient-group} treats principal actions of diffeologically discrete groups and proves Theorem~\ref{mainthm:B}.
Section~\ref{sec:irr} is devoted to irrational tori.

\subsection*{Conventions}
We follow the conventions of \cite{Shio}.
$\Dcat$ is the category of diffeological spaces, $\Elast\subseteq\Dcat$ the full subcategory of elastic spaces, and $T$ Blohmann's tangent functor.
Composition is written in the applicative order, and in long chains we write $gf$ for $g\circ f$.
Maps are smooth unless stated otherwise.
A \emph{path} is a smooth map $\R\to M$.

\subsection*{Acknowledgements}
I would like to express my sincere gratitude to my supervisor, Keiichi Sakai, for his valuable guidance and continuous support since my undergraduate years.
I am also grateful to my sub-advisor, Katsuhiko Kuribayashi, for many insightful discussions and much helpful advice, and to David Miyamoto for helpful correspondence, generous encouragement, and permission to build on the material of \cite{Miy}.

This work was supported by JST SPRING, Grant Number JPMJSP2144, at Shinshu University.

\subsection*{Declaration of generative AI in the writing process}
During the preparation of this manuscript, the author used ChatGPT (OpenAI) and Claude (Anthropic) for language editing, for checking proofs, and for discussing the organization of the material.
All mathematical results, proofs, and scientific conclusions were verified by the author, who takes full responsibility for the content of this publication.

% =====================================================================
\section{Preliminaries}\label{sec:prelim}
% =====================================================================

Sections~\ref{subsec:tangent}--\ref{subsec:qchart} collect the background from \cite{Blo,Shio,Miy}, together with a few elementary consequences and notational conventions needed below.
The new ingredients of this section are Propositions~\ref{prop:etale} and~\ref{prop:descent}, which will be used throughout the rest of the paper.

\subsection{The tangent structure on elastic spaces}\label{subsec:tangent}

We use the tangent structure on $\Elast$ of \cite[Thm.~4.2]{Blo}, as recalled in \cite[\S2]{Shio}: the projection $\pi$, zero section $0$, fibrewise addition $+$, scalar multiplication $\kappa\colon\R\times T\to T$, vertical lift $\lambda\colon T\to T^2$ and canonical flip $\tau\colon T^2\to T^2$.
For $k\ge1$ we write $T_kM=TM\times_M\cdots\times_M TM$ for the $k$-fold fibre product over $\pi_M$, with projections $\pr_i$, and $T^n$ for the $n$-fold iterate of $T$.
Let
\[
  \Fam:=\{T^n\mid n\ge0\}\cup\{T_k\mid k\ge1\},\qquad T^0:=\id .
\]
For $F\in\Fam$ we denote by $\pi^F_M\colon FM\to M$ the projection to the base,
\[
  \pi^{T^0}_M=\id_M,\qquad
  \pi^{T^n}_M=\pi_M\circ\pi_{TM}\circ\cdots\circ\pi_{T^{n-1}M},\qquad
  \pi^{T_k}_M=\pi_M\circ\pr_1 ,
\]
and by $F_xM:=(\pi^F_M)^{-1}(x)$ its fibre over $x\in M$.

We shall use the following properties of $T$ \cite[Prop.~2.5]{Shio}: $T$ preserves finite products and subductions, and $T$ is \emph{local}: for a D-open subset $U\subseteq M$, $T$ of the inclusion is an induction with D-open image $\pi_M^{-1}(U)$ (a tangent vector $[P,u,v]$ with $P(u)\in U$ equals $[P|_{P^{-1}(U)},u,v]$).
Consequently, for $F\in\Fam$ we identify $FU$ with the D-open subset $(\pi^F_M)^{-1}(U)$ of $FM$.

\begin{defn}\label{def:over}
Let $p\colon E\to M$ and $p'\colon E'\to M$ be maps.
A map $f\colon E\to E'$ is a \emph{map over $M$} if $p'\circ f=p$, that is, if $f$ is a morphism in the slice category $\Dcat/M$.
In particular, for $F,F'\in\Fam$ a map $f\colon FM\to F'M$ is over $M$ (with respect to the projections $\pi^F_M$, $\pi^{F'}_M$) if $\pi^{F'}_M\circ f=\pi^F_M$, i.e.\ if $(f,\id_M)$ is a bundle morphism in the sense of \cite[Def.~A.2]{Shio}.
\end{defn}

Spaces such as $\R\times FM$ and $T(T_2M)$ are regarded as spaces over $M$ via $\pi^F_M\circ\pr_2$ and $\pi^{T_2}_M\circ\pi_{T_2M}$.
All structure maps (for instance $\kappa_M\colon\R\times TM\to TM$) are maps over $M$; composites, fibre products and tangent maps of maps over $M$ are again maps over $M$ (for the last, if $f\colon T^nM\to T^mM$ is over $M$, then $\pi^{T^{m+1}}_M\,T(f)=\pi^{T^m}_M\,\pi_{T^mM}\,T(f)=\pi^{T^m}_M\,f\,\pi_{T^nM}=\pi^{T^{n+1}}_M$).
We use this without comment.
A map $f$ over $M$ satisfies $f(FU)\subseteq F'U$ for every D-open $U\subseteq M$, so it restricts to a map $FU\to F'U$ over $U$.

Since $T$ preserves the canonical pullbacks defining the fibre powers $T_k$ (an axiom of tangent categories), the natural map
\begin{equation}\label{eq:chi}
  \chi_M:=\bigl\langle T(\pr_1),T(\pr_2)\bigr\rangle\colon
  T(T_2M)\xrightarrow{\ \cong\ }T^2M\times_{T(\pi_M)}T^2M
\end{equation}
is an isomorphism, natural in $M$.

On $T_2M$ we have the two scalar multiplications
\[
  \kappa^{(1)}(r,(v_0,v_1))=(\kappa_M(r,v_0),v_1),\qquad
  \kappa^{(2)}(r,(v_0,v_1))=(v_0,\kappa_M(r,v_1)).
\]
On $T^2M$, besides $\kappa_{TM}$ (for the bundle $\pi_{TM}$), there is the scalar multiplication for the bundle $T(\pi_M)$, the partial tangent
\[
  T_{(2)}\kappa_M:=T(\kappa_M)\circ(0_\R\times\id_{T^2M})\colon
  \R\times T^2M\to T^2M,
\]
where $T(\R\times TM)\cong T\R\times T^2M$; by \cite[Def.~A.8]{Shio} it satisfies $T_{(2)}\kappa_M=\tau_M\circ\kappa_{TM}\circ(\id_\R\times\tau_M)$.

\begin{rem}[Coordinates]\label{rem:coordinates}
For $M=\R^n$ write points of $T^2M$ as $(u,v_0,v_1,v_{01})$, so that
\[
  \pi_{TM}(u,v_0,v_1,v_{01})=(u,v_0),\qquad T(\pi_M)(u,v_0,v_1,v_{01})=(u,v_1).
\]
Then
\[
  \lambda_M(u,v)=(u,0,0,v),\qquad
  \tau_M(u,v_0,v_1,v_{01})=(u,v_1,v_0,v_{01}),
\]
\[
  \kappa_{TM}\bigl(t,(u,v_0,v_1,v_{01})\bigr)=(u,v_0,tv_1,tv_{01}),\qquad
  T_{(2)}\kappa_M\bigl(t,(u,v_0,v_1,v_{01})\bigr)=(u,tv_0,v_1,tv_{01}).
\]
Note that $T_{(2)}\kappa_M$ scales the base direction $v_0$ of the bundle $\pi_{TM}$, not $v_1$.
\end{rem}

For every map $f\colon M\to N$ of elastic spaces, naturality of the structure maps gives
\begin{align}
  \pi_N\circ Tf&=f\circ\pi_M, \label{eq:nat-pi}\\
  0_N\circ f&=Tf\circ 0_M, \label{eq:nat-zero}\\
  \kappa_N\circ(\id_\R\times Tf)&=Tf\circ\kappa_M, \label{eq:nat-kappa}\\
  +_N\circ(Tf\times_f Tf)&=Tf\circ +_M, \label{eq:nat-plus}\\
  \lambda_N\circ Tf&=T^2f\circ\lambda_M, \label{eq:nat-lambda}\\
  \tau_N\circ T^2f&=T^2f\circ\tau_M. \label{eq:nat-tau}
\end{align}
We shall apply these to $f=\phi$ and to $f=T\phi$; for instance $\lambda_{TY}\circ T^2\phi=T^3\phi\circ\lambda_{TX}$ and $\tau_{TY}\circ T^3\phi=T^3\phi\circ\tau_{TX}$.
In particular each $T_xf\colon T_xM\to T_{f(x)}N$ is $\R$-linear.
Moreover $T_kf\colon T_kM\to T_kN$ acts componentwise, so that it commutes with the projections $\pr_i$ and with $\kappa^{(i)}$, and its restriction to $(T_k)_xM$ is $(T_xf)^{\times k}$.

\begin{lem}\label{lem:nat-T2kappa}
For every $f\colon M\to N$, $T_{(2)}\kappa_N\circ(\id_\R\times T^2f)=T^2f\circ T_{(2)}\kappa_M$.
\end{lem}

\begin{proof}
By $T_{(2)}\kappa=\tau\circ\kappa_T\circ(\id_\R\times\tau)$, \eqref{eq:nat-tau} and \eqref{eq:nat-kappa} applied to $Tf$,
\begin{align*}
  \tau_N\,\kappa_{TN}\,(\id_\R\times\tau_N)(\id_\R\times T^2f)
  &=\tau_N\,\kappa_{TN}\,(\id_\R\times T^2f)(\id_\R\times\tau_M)\\
  &=\tau_N\,T^2f\,\kappa_{TM}\,(\id_\R\times\tau_M)
  =T^2f\,\tau_M\,\kappa_{TM}\,(\id_\R\times\tau_M). \qedhere
\end{align*}
\end{proof}

\subsection{Riemannian structures}\label{subsec:riem}

We recall the notions of \cite{Shio}; $M$ is an elastic space.

\begin{defn}[{\cite[Defs.~3.1, 3.8, 3.18]{Shio}}]\label{def:vconn}
A \emph{vertical connection} on $M$ is a map $K\colon T^2M\to TM$ such that
\begin{enumerate}[label=(\alph*),leftmargin=*]
\item $K\circ\lambda_M=\id_{TM}$;
\item $\pi_M\circ K=\pi_M\circ\pi_{TM}$;
\item $\lambda_M\circ K=T(K)\circ\lambda_{TM}$;
\item $\lambda_M\circ K=T(K)\circ\tau_{TM}\circ T(\lambda_M)$;
\item $K\circ\kappa_{TM}=\kappa_M\circ(\id_\R\times K)$;
\item $K\circ T_{(2)}\kappa_M=\kappa_M\circ(\id_\R\times K)$.
\end{enumerate}
It is \emph{torsion-free} if $K\circ\tau_M=K$, and \emph{effective} if
\[
  \Theta^K_M:=\bigl\langle\pi_{TM},\,T(\pi_M),\,K\bigr\rangle\colon
  T^2M\longrightarrow T_3M
\]
is a diffeomorphism.
The \emph{curvature morphism} of $K$ is $C_K:=K\circ T(K)\circ\tau_{TM}-K\circ T(K)\colon T^3M\to TM$, the difference being taken in the fibres of $\pi_M$ \cite[Def.~5.1]{Shio}; $K$ is \emph{flat} if $C_K=0_M\circ\pi^{T^3}_M$.
\end{defn}

By axiom (b), a vertical connection is a map over $M$ in the sense of Definition~\ref{def:over}.

\begin{defn}[{\cite[Defs.~3.12, 3.15]{Shio}}]\label{def:hconn}
A \emph{horizontal connection} on $M$ is a map $H\colon T_2M\to T^2M$ such that
\begin{enumerate}[label=(\alph*),leftmargin=*]
\item $T(\pi_M)\circ H=\pr_1$;
\item $\pi_{TM}\circ H=\pr_2$;
\item $\lambda_{TM}\circ H=T(H)\circ(\lambda_M\times 0_{TM})$;
\item $\tau_{TM}\circ T(\lambda_M)\circ H=T(H)\circ(0_{TM}\times\lambda_M)$;
\item $H\circ\kappa^{(1)}=\kappa_{TM}\circ(\id_\R\times H)$;
\item $H\circ\kappa^{(2)}=T_{(2)}\kappa_M\circ(\id_\R\times H)$,
\end{enumerate}
where in (c), (d) the pairs take values in $T^2M\times_{T(\pi_M)}T^2M\cong T(T_2M)$ via \eqref{eq:chi}.
A pair $(K,H)$ is a \emph{full connection} if
\[
  \text{(F1)}\ \ K\circ H=0_M\circ\pi_M\circ\pr_1,\qquad
  \text{(F2)}\ \ \mu_M\circ\langle K,\pi_{TM}\rangle+H\circ U_M=\id_{T^2M},
\]
where $U_M:=\langle T(\pi_M),\pi_{TM}\rangle\colon T^2M\to T_2M$,
\[
  \mu_M:=T(+_M)\circ\bigl((\lambda_M\circ\pr_1)\times(0_{TM}\circ\pr_2)\bigr)
  \colon T_2M\to T^2M ,
\]
and the sum in (F2) is taken in the bundle $\pi_{TM}\colon T^2M\to TM$.
\end{defn}

In coordinates, $\mu_M((u,v_0),(u,v_1))=(u,v_1,0,v_0)$.
By \cite[Prop.~3.19, Rem.~3.20]{Shio}, every effective vertical connection $K$ has a unique partner $H$ with $(K,H)$ full, namely $H=(\Theta^K_M)^{-1}\circ\langle\pr_2,\pr_1, 0_M\pi_M\pr_1\rangle$.

\begin{defn}[{\cite[Defs.~4.1, 4.11]{Shio}}]\label{def:metric}
A \emph{Riemannian metric} on $M$ is a map $g\colon T_2M\to\R$ whose restriction to $T_xM\times T_xM$ is a positive definite symmetric bilinear form for every $x\in M$.
A vertical connection $K$ is a \emph{Levi-Civita connection} of $g$ if it is torsion-free and
\begin{equation}\label{eq:LC}
  d(g)=+_\R\circ\bigl\langle g\circ(K\times_\pi\pi_{TM}),\;
  g\circ(\pi_{TM}\times_\pi K)\bigr\rangle ,
\end{equation}
as maps $T(T_2M)\to\R$, where $d(g):=\pr_2\circ T(g)$, $T(T_2M)$ is identified with $T^2M\times_{T(\pi_M)}T^2M$ by \eqref{eq:chi}, and $(K\times_\pi\pi_{TM})(A,B)=(K(A),\pi_{TM}(B))$, $(\pi_{TM}\times_\pi K)(A,B)=(\pi_{TM}(A),K(B))$.
\end{defn}

Let $\partial$ be the unit vector field on $\R$ and, for a path $\gamma$, let $\dot\gamma:=T(\gamma)\circ\partial$.
A path $\gamma$ is a \emph{geodesic} of a vertical connection $K$ if $K\circ T(\dot\gamma)\circ\partial=0_M\circ\gamma$ \cite[Def.~6.9]{Shio}.

Recall from \cite[\S6.1]{Shio} (after \cite{CCL}) that a \emph{parametrized dynamical system} $(N,V,s)$ consists of a vector field $V$ on $N$ and a map $s\colon Z\to N$, and a \emph{solution} is a map $\Phi\colon\R\times Z\to N$ with $\Phi(0,-)=s$ and $T(\Phi)\circ(\partial\times0_Z)=V\circ\Phi$.
When $Z$ is a point, solutions are the \emph{integral curves} of $V$ through $s$.
A vector field $V$ on $N$ is \emph{complete} if $(N,V,\id_N)$ has a solution.
Following \cite[Def.~6.18]{Shio}, $\Elastc$ is the full subcategory of those $M\in\Elast$ such that for every $n\ge0$ every parametrized dynamical system on $T^nM$ has at most one solution.
Every finite-dimensional smooth manifold and every Banach manifold belongs to $\Elastc$ \cite[Rem.~6.20]{Shio}.

The \emph{geodesic spray} of a full connection $(K,H)$ is $S:=H\circ\langle\id_{TM},\id_{TM}\rangle\colon TM\to T^2M$; it is a vector field on $TM$ \cite[Def.~6.16]{Shio}.

\begin{conv}\label{conv:flow}
By the \emph{geodesic flow} of $(K,H)$ we mean a solution $\Phi\colon\R\times TM\to TM$ of the first-order system $(TM,S,\id_{TM})$; it exists if and only if $S$ is complete.
The geodesic flow of \cite[Def.~6.16]{Shio} is $\pi_M\circ\Phi$, which solves the corresponding second-order system by \cite[Prop.~6.14]{Shio}.
\end{conv}

The following lemma makes the relation between geodesics and the spray explicit; its part (iii) is the uniqueness statement of \cite[Thm.~6.21]{Shio}, which in fact does not need completeness.

\begin{lem}\label{lem:geod-spray}
Let $(K,H)$ be a full connection on $M$ with geodesic spray $S$.
\begin{enumerate}[label=\textup{(\roman*)},leftmargin=*]
\item A path $\gamma$ is a geodesic of $K$ if and only if $\dot\gamma$ is an integral curve of $S$.
\item If $c\colon\R\to TM$ is an integral curve of $S$, then $\gamma:=\pi_M\circ c$ is a geodesic and $\dot\gamma=c$.
\item If $M\in\Elastc$, a geodesic is determined by its initial velocity.
\end{enumerate}
\end{lem}

\begin{proof}
(i) Let $W:=T(\dot\gamma)\circ\partial\colon\R\to T^2M$.
Then $\pi_{TM}W=\dot\gamma$ and $T(\pi_M)W=T(\pi_M\dot\gamma)\partial=\dot\gamma$, so $U_MW=(\dot\gamma,\dot\gamma)$ and $HU_MW=S\dot\gamma$.
If $\gamma$ is a geodesic, then $KW=0_M\gamma$, and since $\lambda_M\circ0_M=0_{TM}\circ0_M$ and $T(+_M)\circ 0_{T_2M}=0_{TM}\circ+_M$ we get $\mu_M\langle K,\pi_{TM}\rangle W=0_{TM}\dot\gamma$, the zero of the fibre of $\pi_{TM}$ over $\dot\gamma$; hence (F2) gives $W=S\dot\gamma$.
Conversely, if $W=S\dot\gamma$ then $KW=KH\langle\dot\gamma,\dot\gamma\rangle=0_M\gamma$ by (F1).
(ii) Since $T(\pi_M)\circ S=\id_{TM}$ by (a) of Definition~\ref{def:hconn}, $\dot\gamma=T(\pi_M)T(c)\partial=T(\pi_M)Sc=c$; now apply (i).
(iii) If $\gamma,\gamma'$ are geodesics with $\dot\gamma(0)=\dot\gamma'(0)$, then $\dot\gamma,\dot\gamma'$ are integral curves of $S$ with the same initial value by (i), so they coincide because $M\in\Elastc$ (case $n=1$, $Z$ a point), and $\gamma=\pi_M\dot\gamma=\pi_M\dot\gamma'=\gamma'$.
\end{proof}

\subsection{Q-charts}\label{subsec:qchart}

\begin{defn}[{\cite[Def.~3.14]{Miy}}]\label{def:qchart}
A map $\phi\colon X\to Y$ of diffeological spaces is a \emph{Q-chart} if
\begin{enumerate}[label=(Q\arabic*),leftmargin=*]
\item $\phi$ is a subduction;
\item the diagonal is D-open in $X\times_\phi X=\{(x,x')\mid\phi(x)=\phi(x')\}$;
\item the second (equivalently, the first) projection $X\times_\phi X\to X$ is a local diffeomorphism.
\end{enumerate}
\end{defn}

The fibre product $X\times_\phi X$ is the kernel pair of $\phi$; with source $\pr_2$ and target $\pr_1$ it is the \emph{relation groupoid} $\mathcal{R}_\phi$ of \cite{Miy}, which is \'etale in the sense of \cite[Def.~3.7]{Miy} when $\phi$ is a Q-chart.

\begin{defn}[{\cite[Def.~3.18]{Miy}}]\label{def:transition}
A \emph{transition} of $X$ is a diffeomorphism $h\colon A\to B$ between D-open subsets $A=\dom h$ and $B$ of $X$.
We set
\[
  \Psi(\phi):=\{h \text{ transition of }X\mid \phi\circ h=\phi\text{ on }\dom h\}.
\]
\end{defn}

\begin{lem}[{\cite[Lem.~3.19]{Miy}}]\label{lem:miy-germ}
Let $\phi$ be a Q-chart and $(x,x')\in X\times_\phi X$.
There is $h\in\Psi(\phi)$ with $h(x')=x$, and the germ of such an $h$ at $x'$ is unique.
\end{lem}

\begin{lem}[{\cite[Lem.~3.20, Prop.~3.21, Cors.~3.22, 3.23]{Miy}}]\label{lem:miy-toolkit}
Let $\phi\colon X\to Y$ be a Q-chart and $F\in\Fam$.
\begin{enumerate}[label=\textup{(\roman*)},leftmargin=*]
\item $F\phi$ is a Q-chart; in particular it is a subduction, hence surjective.
\item For every $x\in X$ the restriction $F_x\phi\colon F_xX\to F_{\phi(x)}Y$ is a bijection.
\item If $F\phi(\xi')=F\phi(\xi)$, there is $h\in\Psi(\phi)$ with $Fh(\xi')=\xi$.
\end{enumerate}
\end{lem}

For $F=T^n$, all three assertions are contained in \cite[Cor.~3.22]{Miy} (see also \cite[Lem.~3.20, Prop.~3.21]{Miy}).
For $F=T_k$, part (i) is \cite[Cor.~3.23]{Miy}, part (ii) follows from \cite[Lem.~3.20]{Miy} since $(T_k)_x\phi=(T_x\phi)^{\times k}$, and part (iii) is proved as follows.

\begin{proof}[Proof of \textup{(iii)} for $F=T_k$]
Let $\xi=(w_1,\dots,w_k)$, $\xi'=(w_1',\dots,w_k')$ with $T_k\phi(\xi)=T_k\phi(\xi')$, and let $x,x'$ be their base points.
By the case $F=T$ there are $h_i\in\Psi(\phi)$ with $Th_i(w_i')=w_i$, whence $h_i(x')=x$.
By Lemma~\ref{lem:miy-germ} all $h_i$ have the same germ at $x'$, so they agree on a D-open neighbourhood $U$ of $x'$, and by locality of $T$ the maps $Th_i$ agree on $TU\supseteq T_{x'}X$.
Hence $h:=h_1$ satisfies $T_kh(\xi')=\xi$.
\end{proof}

\subsection{Q-charts and tangent-level pullbacks}\label{subsec:etale}

The following observation is the basic tool of this paper.

\begin{prop}\label{prop:etale}
Let $\phi\colon X\to Y$ be a Q-chart with $X$ elastic, and let $F\in\Fam$.
Then $Y$ is elastic by \cite[Thm.~3.24]{Miy}, and
\[
  \Pi^F_\phi:=\bigl(\pi^F_X,\,F\phi\bigr)\colon FX\longrightarrow
  X\times_Y FY=\{(x,\eta)\mid\phi(x)=\pi^F_Y(\eta)\}
\]
is a diffeomorphism.
In other words, the square formed by $F\phi$, $\pi^F_X$, $\pi^F_Y$ and $\phi$ is a pullback in $\Dcat$.
\end{prop}

\begin{proof}
The map $\Pi^F_\phi$ is well defined by naturality of the projections and bijective by Lemma~\ref{lem:miy-toolkit}(ii).
It remains to show that its inverse is smooth, and since smoothness is a local property of plots it suffices to lift plots locally.

\emph{The case $F=T$.} Let $(Q,R)\colon U\to X\times_Y TY$ be a plot and $u_0\in U$.
Since $T\phi$ is a subduction, after shrinking $U$ there is a plot $R'\colon U\to TX$ with $T\phi\circ R'=R$.
Then $\phi(\pi_XR'(u))=\pi_YR(u)= \phi(Q(u))$, so by Lemma~\ref{lem:miy-germ} there is $h\in\Psi(\phi)$ with $h(\pi_XR'(u_0))=Q(u_0)$.
After shrinking $U$ further, $R'(U)\subseteq T(\dom h)$ and $\tilde R:=Th\circ R'$ is a plot of $TX$ with $T\phi\circ\tilde R=T(\phi h)\circ R'=R$.
The plot $u\mapsto\bigl(Q(u),\,\pi_X\tilde R(u)\bigr)=\bigl(Q(u),h(\pi_XR'(u))\bigr)$ of $X\times_\phi X$ sends $u_0$ into the diagonal, which is D-open by (Q2); hence $\pi_X\tilde R=Q$ near $u_0$.
Thus $(Q,R)=\Pi^T_\phi\circ\tilde R$ near $u_0$.

\emph{The case $F=T^n$.} By induction, using that $T^{n-1}\phi$ is a Q-chart (Lemma~\ref{lem:miy-toolkit}(i)) and the case $F=T$ applied to it,
\[
  T^nX\cong T^{n-1}X\times_{T^{n-1}Y}T^nY
  \cong\bigl(X\times_Y T^{n-1}Y\bigr)\times_{T^{n-1}Y}T^nY
  \cong X\times_Y T^nY ,
\]
and the composite of these diffeomorphisms is $\Pi^{T^n}_\phi$.

\emph{The case $F=T_k$.} The inverse is $(x,(w_1,\dots,w_k))\mapsto\bigl((\Pi^T_\phi)^{-1}(x,w_1),\dots, (\Pi^T_\phi)^{-1}(x,w_k)\bigr)$, which is smooth by the case $F=T$.
\end{proof}

\begin{rem}\label{rem:not-local-diffeo}
A Q-chart need not be a local diffeomorphism: the quotient map $q\colon\R\to\R/(\Z+\alpha\Z)$ of Section~\ref{sec:irr} is injective on no nonempty open set.
Proposition~\ref{prop:etale} says that, with respect to the tangent functor, a Q-chart nevertheless behaves like a local diffeomorphism.
\end{rem}

\begin{cor}[Detection]\label{cor:detect}
Let $\phi$ be as in Proposition~\ref{prop:etale}, $F\in\Fam$, and let $a,b\colon Z\to FX$ be maps with $\pi^F_X\circ a=\pi^F_X\circ b$ and $F\phi\circ a=F\phi\circ b$.
Then $a=b$.
\end{cor}

\begin{proof}
$\Pi^F_\phi\circ a=\Pi^F_\phi\circ b$, and $\Pi^F_\phi$ is injective.
\end{proof}

\subsection{Descent and lifting}\label{subsec:descent}

Recall that the subductions are the strong epimorphisms of $\Dcat$ \cite[\S2]{Miy}.
We use the following standard fact.

\begin{fact}\label{fact:subduction-coeq}
Subductions in $\Dcat$ are effective epimorphisms: a subduction $q\colon A\to B$ is the coequalizer of its kernel pair $A\times_BA\rightrightarrows A$.
Consequently, if $f\colon A\to Z$ is constant on the fibres of a subduction $q$, there is a unique smooth map $\bar f\colon B\to Z$ with $\bar f\circ q=f$; indeed, $\bar f$ exists as a map of sets, and $\bar f\circ P$ is smooth for every plot $P$ of $B$ because $P$ lifts locally along $q$.
\end{fact}

\begin{defn}\label{def:psi-invariant}
Let $\phi\colon X\to Y$ be a Q-chart and $F,F'\in\Fam$.
A map $f\colon FX\to F'X$ over $X$ is \emph{$\Psi(\phi)$-invariant} if
\[
  f\circ Fh=F'h\circ f\quad\text{on }F(\dom h)\qquad\text{for all }h\in\Psi(\phi).
\]
A map $f\colon FX\to Z$ to a diffeological space $Z$ is \emph{$\Psi(\phi)$-invariant} if $f\circ Fh=f$ on $F(\dom h)$ for all $h\in\Psi(\phi)$.
\end{defn}

Both sides of the first equation make sense because $f$ is a map over $X$, so that $f(F(\dom h))\subseteq F'(\dom h)$ (Definition~\ref{def:over}).

The following proposition is the descent principle used throughout the paper: $\Psi(\phi)$-invariance is necessary and sufficient for descent along $F\phi$, every map over $Y$ has a unique invariant lift, and descent respects composition.

\begin{prop}[Descent and lifting]\label{prop:descent}
Let $\phi\colon X\to Y$ be a Q-chart with $X$ elastic (hence $Y$ elastic), and let $F,F'\in\Fam$.
\begin{enumerate}[label=\textup{(\roman*)},leftmargin=*]
\item A map $f\colon FX\to Z$ is of the form $f=f_Y\circ F\phi$ for some map $f_Y\colon FY\to Z$ if and only if $f$ is $\Psi(\phi)$-invariant; in this case $f_Y$ is unique.
\item Let $f\colon FX\to F'X$ be a map over $X$.
  There exists a map $f_Y\colon FY\to F'Y$ over $Y$ with
  \[
    F'\phi\circ f=f_Y\circ F\phi
  \]
if and only if $f$ is $\Psi(\phi)$-invariant; in this case $f_Y$ is unique.
Conversely, every map $f_Y\colon FY\to F'Y$ over $Y$ arises in this way from a unique $\Psi(\phi)$-invariant map $f$ over $X$, namely
  \[
    f=(\Pi^{F'}_\phi)^{-1}\circ\bigl\langle\pi^F_X,\;f_Y\circ F\phi\bigr\rangle .
  \]
We write $f\leftrightarrow f_Y$ for this correspondence.
\item The correspondence of \textup{(ii)} is compatible with composition and identities: if $f\leftrightarrow f_Y$ and $f'\leftrightarrow f'_Y$ for composable $f\colon FX\to F'X$, $f'\colon F'X\to F''X$, then $f'\circ f\leftrightarrow f'_Y\circ f_Y$, and $\id_{FX}\leftrightarrow\id_{FY}$.
  In particular, $f$ is a diffeomorphism if and only if $f_Y$ is.
\end{enumerate}
\end{prop}

\begin{proof}
(i) If $f=f_Y\circ F\phi$, then $f\circ Fh=f_Y\circ F(\phi h)=f$, so $f$ is invariant.
Conversely, let $f$ be invariant.
If $F\phi(\xi)=F\phi(\xi')$, then by Lemma~\ref{lem:miy-toolkit}(iii) there is $h\in\Psi(\phi)$ with $Fh(\xi')=\xi$, hence $f(\xi)=f(Fh(\xi'))=f(\xi')$.
Since $F\phi$ is a subduction, Fact~\ref{fact:subduction-coeq} gives a unique $f_Y$ with $f_Y\circ F\phi=f$.
Uniqueness in both directions follows from surjectivity of $F\phi$.

(ii) If $f_Y$ over $Y$ with $F'\phi\circ f=f_Y\circ F\phi$ exists, then for $h\in\Psi(\phi)$ the maps $f\circ Fh$ and $F'h\circ f$ lie over $h\circ\pi^F_X$ and have the same image $f_Y\circ F\phi$ under $F'\phi$, so they coincide by Corollary~\ref{cor:detect}; thus $f$ is invariant.
Conversely, let $f$ be invariant and over $X$.
Then $F'\phi\circ f$ is invariant in the sense of (i), because $F'\phi\circ f\circ Fh=F'\phi\circ F'h\circ f= F'(\phi h)\circ f=F'\phi\circ f$, so by (i) there is a unique $f_Y$ with $f_Y\circ F\phi=F'\phi\circ f$; it is a map over $Y$ since $\pi^{F'}_Y f_Y F\phi= \phi\,\pi^{F'}_Xf=\phi\,\pi^F_X=\pi^F_Y F\phi$ and $F\phi$ is surjective.
Finally, given $f_Y$ over $Y$, the map
\[
  f:=(\Pi^{F'}_\phi)^{-1}\circ\bigl\langle\pi^F_X,\;f_Y\circ F\phi\bigr\rangle
  \colon FX\to F'X
\]
is well defined and smooth by Proposition~\ref{prop:etale}, is a map over $X$, and satisfies $F'\phi\circ f=f_Y\circ F\phi$.
It is invariant: for $h\in\Psi(\phi)$ both $f\circ Fh$ and $F'h\circ f$ lie over $h\circ\pi^F_X$ and have the same image $f_Y\circ F\phi$ under $F'\phi$, so they coincide by Corollary~\ref{cor:detect}.
The two constructions are mutually inverse by surjectivity of $F\phi$ and by Corollary~\ref{cor:detect}.

(iii) The composite $f'\circ f$ is again an invariant map over $X$, and $F''\phi\circ f'\circ f=f'_Y\circ F'\phi\circ f=f'_Y\circ f_Y\circ F\phi$; clearly $\id_{FX}\leftrightarrow\id_{FY}$.
If $f$ is a diffeomorphism, then $f^{-1}$ is a map over $X$ (hence maps $F'(\dom h)$ onto $F(\dom h)$) and is invariant, since $f\circ Fh=F'h\circ f$ gives $f^{-1}\circ F'h=Fh\circ f^{-1}$; if $f^{-1}\leftrightarrow g_Y$, then $g_Y\circ f_Y$ and $f_Y\circ g_Y$ correspond to identities, so $g_Y=f_Y^{-1}$ by injectivity of the correspondence.
The converse is the same argument applied to the lift of $f_Y^{-1}$.
\end{proof}

\begin{rem}\label{rem:descent-general}
We shall also use that the maps
\[
  \id_\R\times F\phi\colon\R\times FX\to\R\times FY,\qquad
  T(T_2\phi)\colon T(T_2X)\to T(T_2Y)
\]
are subductions, hence surjective: the first because products of subductions are subductions, the second by Lemma~\ref{lem:miy-toolkit}(i) and because $T$ preserves subductions \cite[Prop.~2.5]{Shio}.
\end{rem}

% =====================================================================
\section{Descent of Riemannian structures along Q-charts}\label{sec:transport}
% =====================================================================

Throughout this section $\phi\colon X\to Y$ is a Q-chart and $X$ is elastic; then $Y$ is elastic by \cite[Thm.~3.24]{Miy}.

\begin{thm}\label{thm:transport}
For a $\Psi(\phi)$-invariant structure on $X$ we write the subscript $Y$ for the corresponding structure on $Y$ given by \textup{(1)--(3)} below.
\begin{enumerate}[label=\textup{(\arabic*)},leftmargin=*]
\item \textup{(Vertical connections)} For every $\Psi(\phi)$-invariant vertical connection $K$ on $X$ there is a unique vertical connection $K_Y$ on $Y$ such that
  \begin{equation}\label{eq:K-rel}
    K_Y\circ T^2\phi=T\phi\circ K ,
  \end{equation}
and every vertical connection on $Y$ arises in this way from a unique $\Psi(\phi)$-invariant $K$.
Moreover
  \[
    \Theta^{K_Y}_Y\circ T^2\phi=T_3\phi\circ\Theta^K_X,\qquad
    C_{K_Y}\circ T^3\phi=T\phi\circ C_K ,
  \]
and $K$ is torsion-free \textup{(}resp.\ effective, flat\textup{)} if and only if $K_Y$ is.
\item \textup{(Horizontal and full connections)} For every $\Psi(\phi)$-invariant horizontal connection $H$ on $X$ there is a unique horizontal connection $H_Y$ on $Y$ such that
  \begin{equation}\label{eq:H-rel}
    H_Y\circ T_2\phi=T^2\phi\circ H ,
  \end{equation}
and every horizontal connection on $Y$ arises in this way from a unique $\Psi(\phi)$-invariant $H$.
For corresponding $K\leftrightarrow K_Y$ and $H\leftrightarrow H_Y$, the pair $(K,H)$ is a full connection if and only if $(K_Y,H_Y)$ is.
If $K$ is $\Psi(\phi)$-invariant and effective, its partner $H$ is $\Psi(\phi)$-invariant and $H_Y$ is the partner of $K_Y$.
\item \textup{(Metrics)} For every $\Psi(\phi)$-invariant Riemannian metric $g$ on $X$ there is a unique Riemannian metric $g_Y$ on $Y$ such that $g=g_Y\circ T_2\phi$, and every Riemannian metric on $Y$ arises in this way from a unique $\Psi(\phi)$-invariant $g$.
  Every $T_x\phi\colon(T_xX,g)\to (T_{\phi(x)}Y,g_Y)$ is a linear isometry.
\item \textup{(Levi-Civita connections)} For corresponding $K\leftrightarrow K_Y$ and $g\leftrightarrow g_Y$, $K$ is a Levi-Civita connection of $g$ if and only if $K_Y$ is a Levi-Civita connection of $g_Y$.
\item \textup{(Geodesics)} For corresponding $K\leftrightarrow K_Y$, a path $\gamma\colon\R\to X$ is a geodesic of $K$ if and only if $\phi\circ\gamma$ is a geodesic of $K_Y$.
\end{enumerate}
\end{thm}

Existence, uniqueness and the converse lifting statements follow from Proposition~\ref{prop:descent}.
It remains to show that the defining identities of the respective structures are preserved under descent and lifting.
We shall repeatedly use the following principle.
In it, $D$ denotes any one of the functors
\[
  F,\qquad \R\times F,\qquad T\circ T_2\qquad(F\in\Fam);
\]
in each case $D\phi$ is a subduction, hence surjective (Lemma~\ref{lem:miy-toolkit}(i) and Remark~\ref{rem:descent-general}).

\begin{lem}[Transfer]\label{lem:transfer}
Let $D$ be as above and $F'\in\Fam$.
Suppose that, for $M=X,Y$, maps $a_M,b_M\colon DM\to F'M$ satisfy
\[
  F'\phi\circ a_X=a_Y\circ D\phi,\qquad F'\phi\circ b_X=b_Y\circ D\phi,
  \qquad \pi^{F'}_X\circ a_X=\pi^{F'}_X\circ b_X .
\]
Then $a_X=b_X$ if and only if $a_Y=b_Y$.
The analogous statement holds for maps $a_M,b_M\colon DM\to\R$ with $a_X=a_Y\circ D\phi$ and $b_X=b_Y\circ D\phi$.
\end{lem}

\begin{proof}
If $a_Y=b_Y$, then $F'\phi\circ a_X=F'\phi\circ b_X$, and $a_X=b_X$ by Corollary~\ref{cor:detect}.
If $a_X=b_X$, then $a_Y\circ D\phi=b_Y\circ D\phi$, and $a_Y=b_Y$ since $D\phi$ is surjective.
In the real-valued case, $a_Y=b_Y$ clearly implies $a_X=b_X$, and the converse follows from the surjectivity of $D\phi$.
\end{proof}

In all applications below, $a_M$ and $b_M$ are built from the structure maps, from the maps $K$, $H$ over $M$ and from $T$, so the base-point condition $\pi^{F'}_X a_X=\pi^{F'}_X b_X$ holds automatically (\S\ref{subsec:tangent}).
The intertwining identities are consequences of \eqref{eq:nat-pi}--\eqref{eq:nat-tau}, Lemma~\ref{lem:nat-T2kappa}, the defining relations \eqref{eq:K-rel}, \eqref{eq:H-rel}, and their images under $T$,
\begin{equation}\label{eq:TK-rel}
  T(K_Y)\circ T^3\phi=T^2\phi\circ T(K),\qquad
  T(H_Y)\circ T(T_2\phi)=T^3\phi\circ T(H).
\end{equation}

\subsection{Vertical connections}\label{subsec:proof-K}

\begin{proof}[Proof of Theorem~\ref{thm:transport}\textup{(1)}]
Axiom (b) says precisely that $K\colon T^2M\to TM$ is a map over $M$.
By Proposition~\ref{prop:descent}(ii) with $F=T^2$, $F'=T$, the relation \eqref{eq:K-rel} is a bijection between $\Psi(\phi)$-invariant maps $K\colon T^2X\to TX$ over $X$ and maps $K_Y\colon T^2Y\to TY$ over $Y$.
By Lemma~\ref{lem:transfer}, it remains only to verify that the two sides of axioms (a), (c)--(f) satisfy the required intertwining relations.

(a) Here $D=T$: $\;T\phi\,K\,\lambda_X=K_Y\,T^2\phi\,\lambda_X=K_Y\,\lambda_Y\,T\phi$ and $T\phi\,\id_{TX}=\id_{TY}\,T\phi$.

(c) Here $D=F'=T^2$.
On the left, $T^2\phi\,\lambda_X\,K=\lambda_Y\,T\phi\,K= \lambda_Y\,K_Y\,T^2\phi$.
On the right, by \eqref{eq:TK-rel} and \eqref{eq:nat-lambda} for $f=T\phi$,
\[
  T^2\phi\,T(K)\,\lambda_{TX}=T(K_Y)\,T^3\phi\,\lambda_{TX}
  =T(K_Y)\,\lambda_{TY}\,T^2\phi .
\]

(d) The left-hand side is as in (c).
On the right,
\begin{align*}
  T^2\phi\,T(K)\,\tau_{TX}\,T(\lambda_X)
  &=T(K_Y)\,T^3\phi\,\tau_{TX}\,T(\lambda_X)
   =T(K_Y)\,\tau_{TY}\,T^3\phi\,T(\lambda_X)\\
  &=T(K_Y)\,\tau_{TY}\,T(T^2\phi\,\lambda_X)
   =T(K_Y)\,\tau_{TY}\,T(\lambda_Y\,T\phi)
   =T(K_Y)\,\tau_{TY}\,T(\lambda_Y)\,T^2\phi .
\end{align*}

(e) Here $D=\R\times T^2$.
By \eqref{eq:nat-kappa} for $f=T\phi$ and $f=\phi$,
\begin{align*}
  T\phi\,K\,\kappa_{TX}&=K_Y\,T^2\phi\,\kappa_{TX}
   =K_Y\,\kappa_{TY}\,(\id_\R\times T^2\phi),\\
  T\phi\,\kappa_X\,(\id_\R\times K)&=\kappa_Y\,(\id_\R\times T\phi\,K)
   =\kappa_Y\,(\id_\R\times K_Y)\,(\id_\R\times T^2\phi).
\end{align*}

(f) As in (e), with $\kappa_{TM}$ replaced by $T_{(2)}\kappa_M$ and \eqref{eq:nat-kappa} by Lemma~\ref{lem:nat-T2kappa}.

\emph{Torsion.} With $D=T^2$, $F'=T$: $T\phi\,K\,\tau_X=K_Y\,T^2\phi\,\tau_X= K_Y\,\tau_Y\,T^2\phi$ and $T\phi\,K=K_Y\,T^2\phi$.
Hence $K\tau_X=K$ if and only if $K_Y\tau_Y=K_Y$.

\emph{Curvature.} By \eqref{eq:K-rel}, \eqref{eq:TK-rel} and \eqref{eq:nat-tau} for $f=T\phi$,
\[
  T\phi\,K\,T(K)=K_Y\,T(K_Y)\,T^3\phi,\qquad
  T\phi\,K\,T(K)\,\tau_{TX}=K_Y\,T(K_Y)\,\tau_{TY}\,T^3\phi .
\]
Since $T\phi$ is additive on fibres \eqref{eq:nat-plus}, subtracting gives $C_{K_Y}\circ T^3\phi=T\phi\circ C_K$.
As $T\phi\,0_X\,\pi^{T^3}_X=0_Y\,\pi^{T^3}_Y\,T^3\phi$, Lemma~\ref{lem:transfer} shows that $K$ is flat if and only if $K_Y$ is.

\emph{Effectivity.} By \eqref{eq:nat-pi} for $f=T\phi$, the naturality $T(\pi_Y)\,T^2\phi=T\phi\,T(\pi_X)$, and \eqref{eq:K-rel}, the three components of $\Theta$ intertwine, so $\Theta^{K_Y}_Y\circ T^2\phi=T_3\phi\circ\Theta^K_X$.
Moreover $\Theta^K_X$ is a $\Psi(\phi)$-invariant map over $X$, because its components are.
Thus $\Theta^K_X$ and $\Theta^{K_Y}_Y$ correspond under Proposition~\ref{prop:descent}(ii), and by Proposition~\ref{prop:descent}(iii) one is a diffeomorphism if and only if the other is.
\end{proof}

\subsection{Horizontal and full connections}\label{subsec:proof-H}

\begin{proof}[Proof of Theorem~\ref{thm:transport}\textup{(2)}]
By axiom (b), $\pi^{T^2}_M\circ H=\pi_M\circ\pr_2=\pi^{T_2}_M$, so $H$ is a map over $M$, and Proposition~\ref{prop:descent}(ii) with $F=T_2$, $F'=T^2$ gives the bijection of maps over $X$ and over $Y$.
We check the intertwining identities.
The projections satisfy $\pr_i\circ T_2\phi=T\phi\circ\pr_i$, and by naturality of $\chi$ \eqref{eq:chi} the map $T(T_2\phi)$ corresponds to $T^2\phi\times_{T\phi}T^2\phi$.
Hence \eqref{eq:nat-lambda} for $f=\phi$ and \eqref{eq:nat-zero} for $f=T\phi$ give
\begin{equation}\label{eq:sharp}
  (\lambda_Y\times0_{TY})\circ T_2\phi=T(T_2\phi)\circ(\lambda_X\times0_{TX}),
  \qquad
  (0_{TY}\times\lambda_Y)\circ T_2\phi=T(T_2\phi)\circ(0_{TX}\times\lambda_X).
\end{equation}

(a), (b) With $D=T_2$, $F'=T$: $T\phi\,T(\pi_X)\,H=T(\pi_Y)\,T^2\phi\,H= T(\pi_Y)\,H_Y\,T_2\phi$ and $T\phi\,\pr_1=\pr_1\,T_2\phi$; similarly $T\phi\,\pi_{TX}\,H=\pi_{TY}\,H_Y\,T_2\phi$ and $T\phi\,\pr_2=\pr_2\,T_2\phi$.

(c) With $D=T_2$, $F'=T^3$: on the left $T^3\phi\,\lambda_{TX}\,H=\lambda_{TY}\,T^2\phi\,H=\lambda_{TY}\,H_Y\,T_2\phi$; on the right, by \eqref{eq:TK-rel} and \eqref{eq:sharp},
\[
  T^3\phi\,T(H)\,(\lambda_X\times0_{TX})
  =T(H_Y)\,T(T_2\phi)\,(\lambda_X\times0_{TX})
  =T(H_Y)\,(\lambda_Y\times0_{TY})\,T_2\phi .
\]

(d) On the left, $T^3\phi\,\tau_{TX}\,T(\lambda_X)\,H=\tau_{TY}\,T^3\phi\,T(\lambda_X)\,H =\tau_{TY}\,T(\lambda_Y)\,T^2\phi\,H=\tau_{TY}\,T(\lambda_Y)\,H_Y\,T_2\phi$; on the right one argues as in (c) with the second identity of \eqref{eq:sharp}.

(e), (f) With $D=\R\times T_2$, using $\kappa^{(i)}_Y\circ(\id_\R\times T_2\phi)=T_2\phi\circ\kappa^{(i)}_X$, \eqref{eq:nat-kappa} for $f=T\phi$ and Lemma~\ref{lem:nat-T2kappa}:
\begin{align*}
  T^2\phi\,H\,\kappa^{(1)}_X&=H_Y\,\kappa^{(1)}_Y\,(\id_\R\times T_2\phi), &
  T^2\phi\,\kappa_{TX}\,(\id_\R\times H)&=\kappa_{TY}\,(\id_\R\times H_Y)\,
    (\id_\R\times T_2\phi),\\
  T^2\phi\,H\,\kappa^{(2)}_X&=H_Y\,\kappa^{(2)}_Y\,(\id_\R\times T_2\phi), &
  T^2\phi\,T_{(2)}\kappa_X\,(\id_\R\times H)&=T_{(2)}\kappa_Y\,(\id_\R\times H_Y)\,
    (\id_\R\times T_2\phi).
\end{align*}

\emph{Full connections.} For (F1), with $D=T_2$ and $F'=T$, $T\phi\,K\,H=K_Y\,T^2\phi\,H=K_Y\,H_Y\,T_2\phi$ and $T\phi\,0_X\,\pi_X\,\pr_1=0_Y\,\phi\,\pi_X\,\pr_1=0_Y\,\pi_Y\,\pr_1\,T_2\phi$.
For (F2), with $D=F'=T^2$, we first note
\[
  \mu_Y\circ T_2\phi=T^2\phi\circ\mu_X,\qquad
  U_Y\circ T^2\phi=T_2\phi\circ U_X,\qquad
  \langle K_Y,\pi_{TY}\rangle\circ T^2\phi=T_2\phi\circ\langle K,\pi_{TX}\rangle .
\]
The first follows from the naturality of $\lambda$ and $0$ (\eqref{eq:nat-lambda} for $f=\phi$ and \eqref{eq:nat-zero} for $f=T\phi$), together with that of the fibrewise addition, $+_Y\circ T_2\phi=T\phi\circ+_X$ \eqref{eq:nat-plus}; applying $T$ to the latter gives
\[
  T(+_Y)\circ T(T_2\phi)=T^2\phi\circ T(+_X).
\]
The second and third identities follow from \eqref{eq:nat-pi} and \eqref{eq:K-rel}.
The sum in (F2) is taken in the bundle $\pi_{TM}$, and $T^2\phi=T(T\phi)$ is additive for this bundle by \eqref{eq:nat-plus} applied to $f=T\phi$.
Hence
\[
  T^2\phi\circ\bigl(\mu_X\langle K,\pi_{TX}\rangle+H\,U_X\bigr)
  =\bigl(\mu_Y\langle K_Y,\pi_{TY}\rangle+H_Y\,U_Y\bigr)\circ T^2\phi,
\]
and trivially $T^2\phi\circ\id=\id\circ T^2\phi$.

\emph{Partners.} Let $K$ be $\Psi(\phi)$-invariant and effective, with partner $H$.
By (a), (b) and (F1), $\Theta^K_X\circ H=\sigma_X$ with $\sigma_M:=\langle\pr_2,\pr_1,0_M\pi_M\pr_1\rangle\colon T_2M\to T_3M$, so $H=(\Theta^K_X)^{-1}\circ\sigma_X$.
By the proof of (1), $\Theta^K_X$ is a $\Psi(\phi)$-invariant diffeomorphism over $X$, so its inverse is $\Psi(\phi)$-invariant as well (proof of Proposition~\ref{prop:descent}(iii)).
Since $\sigma_X$ also commutes with the tangent maps of all $h\in\Psi(\phi)$, $H$ is $\Psi(\phi)$-invariant.
By what we have shown, $(K_Y,H_Y)$ is a full connection, so $H_Y$ is the partner of $K_Y$ by uniqueness \cite[Prop.~3.19]{Shio}.
\end{proof}

\subsection{Metrics}\label{subsec:proof-g}

\begin{proof}[Proof of Theorem~\ref{thm:transport}\textup{(3)}]
Proposition~\ref{prop:descent}(i) with $F=T_2$ and $Z=\R$ gives a bijection $g\mapsto g_Y$ between $\Psi(\phi)$-invariant maps $T_2X\to\R$ and maps $T_2Y\to\R$.
The Riemannian metric conditions are fibrewise.
For $x\in X$ and $y=\phi(x)$ we have
\[
  g|_{(T_2)_xX}=g_Y|_{(T_2)_yY}\circ(T_x\phi\times T_x\phi),
\]
and $T_x\phi$ is an $\R$-linear bijection (\S\ref{subsec:tangent} and Lemma~\ref{lem:miy-toolkit}(ii)).
Hence $g$ is symmetric, bilinear and positive definite on $T_xX$ if and only if $g_Y$ is so on $T_yY$, and $T_x\phi$ is an isometry.
Since $\phi$ is surjective, every fibre $T_yY$ occurs in this way.
\end{proof}

\begin{rem}[Definiteness in the sense of \cite{KSS}]\label{rem:KSS-definite}
Suppose that $g$ is definite in the sense of \cite[Def.~3.7]{KSS}, i.e.\ there is a generating family $\mathcal{G}_X$ of the diffeology of $X$ such that the pullback tensor $g(P)$ \cite[Notation~4.5]{Shio} is definite for every $P\in\mathcal{G}_X$.
Then $g_Y$ is definite in the same sense.
Indeed, since $\phi$ is a subduction, every plot of $Y$ is locally of the form $\phi\circ P$ with $P$ a plot of $X$, hence locally of the form $\phi\circ P\circ f$ with $P\in\mathcal{G}_X$, or constant; so $\mathcal{G}_Y:=\{\phi\circ P\mid P\in\mathcal{G}_X\}$ generates the diffeology of $Y$, and $g_Y(\phi\circ P)=g(P)$ because $g_Y\circ T_2\phi=g$.
\end{rem}

\subsection{Levi-Civita connections and geodesics}\label{subsec:proof-LC}

\begin{proof}[Proof of Theorem~\ref{thm:transport}\textup{(4)}]
Torsion-freeness transfers by (1).
For the metric compatibility \eqref{eq:LC} we use Lemma~\ref{lem:transfer} with $D=T\circ T_2$ in its real-valued form.
Applying $T$ to $g=g_Y\circ T_2\phi$ and composing with $\pr_2\colon T\R\to\R$ gives $d(g)=d(g_Y)\circ T(T_2\phi)$.
Under the identifications $\chi_X,\chi_Y$ of \eqref{eq:chi}, $T(T_2\phi)$ corresponds to $T^2\phi\times_{T\phi}T^2\phi$, and by \eqref{eq:K-rel} and \eqref{eq:nat-pi},
\[
  (K_Y\times_\pi\pi_{TY})\circ(T^2\phi\times_{T\phi}T^2\phi)
  =\langle T\phi\,K,\,T\phi\,\pi_{TX}\rangle
  =T_2\phi\circ(K\times_\pi\pi_{TX}),
\]
and likewise for $\pi_{TY}\times_\pi K_Y$.
Composing with $g_Y$ and using $g_Y\circ T_2\phi=g$ we find that the right-hand side of \eqref{eq:LC} for $(X,g,K)$ is the right-hand side for $(Y,g_Y,K_Y)$ composed with $T(T_2\phi)$.
Hence \eqref{eq:LC} holds on $X$ if and only if it holds on $Y$.
\end{proof}

\begin{proof}[Proof of Theorem~\ref{thm:transport}\textup{(5)}]
Put $\gamma_Y:=\phi\circ\gamma$.
By functoriality, $\dot\gamma_Y=T\phi\circ\dot\gamma$ and hence, by \eqref{eq:K-rel},
\[
  T\phi\circ K\circ T(\dot\gamma)\circ\partial
  =K_Y\circ T^2\phi\circ T(\dot\gamma)\circ\partial
  =K_Y\circ T(\dot\gamma_Y)\circ\partial ,
  \qquad
  T\phi\circ0_X\circ\gamma=0_Y\circ\gamma_Y .
\]
The maps $K\,T(\dot\gamma)\,\partial$ and $0_X\gamma$ both lie over $\gamma$.
By Corollary~\ref{cor:detect}, $K\,T(\dot\gamma)\,\partial=0_X\gamma$ if and only if $K_Y\,T(\dot\gamma_Y)\,\partial=0_Y\gamma_Y$.
\end{proof}

\begin{rem}\label{rem:transport-scope}
Theorem~\ref{thm:transport} identifies the Riemannian structures on $Y$ considered in \cite{Shio} with the corresponding $\Psi(\phi)$-invariant structures on $X$.
It is not a global lifting statement for paths or for solutions of differential equations; in particular, it does not imply that every geodesic of $Y$ lifts to a geodesic of $X$.
Such statements require global path lifting along $\phi$, which is available for the Q-charts of the next section but need not hold for a general Q-chart: for the Q-chart $(-\infty,1)\sqcup(0,\infty)\to\R$ given by the inclusions, the geodesic $t\mapsto t$ of $dt^2$ on $\R$ has no lift to a path in $(-\infty,1)\sqcup(0,\infty)$.
\end{rem}

% =====================================================================
\section{Quotients by diffeologically discrete groups}\label{sec:quotient-group}
% =====================================================================

In this section $\Lambda$ is a \emph{diffeologically discrete} group, i.e.\ a group equipped with the discrete diffeology (every plot is locally constant), which acts smoothly on a diffeological space $X$ from the left, $(\eta,x)\mapsto\eta x$, and $\phi\colon X\to Y:=X/\Lambda$ denotes the quotient map, $Y$ carrying the quotient diffeology.

\subsection{Principal actions}\label{subsec:principal}

\begin{defn}\label{def:principal-action}
The action is \emph{principal} if it is free and the \emph{division map}
\[
  \divi\colon X\times_\phi X\to\Lambda,\qquad \divi(\eta x,x)=\eta,
\]
is smooth.
\end{defn}

Equivalently, the map
\[
  \Lambda\times X\longrightarrow X\times_\phi X,\qquad(\eta,x)\longmapsto(\eta x,x),
\]
is a diffeomorphism, with inverse $(x',x)\mapsto(\divi(x',x),x)$; that is, the characteristic map $\Lambda\times X\to X\times X$ is an induction.
This is the notion of principal action of \cite[art.~8.11]{PIZ13}, for which $\phi$ is a principal fibration with structure group $\Lambda$.
Since $\Lambda$ is discrete, $\phi$ is, when $Y$ is connected, a covering in the sense of \cite[art.~8.22]{PIZ13}, that is, a fibre bundle with discrete fibre.

\begin{lem}\label{lem:principal-qchart}
Let the action be principal.
For $\eta\in\Lambda$ put $\Gamma_\eta:=\{(\eta x,x)\mid x\in X\}=\divi^{-1}(\eta)$.
\begin{enumerate}[label=\textup{(\roman*)},leftmargin=*]
\item Each $\Gamma_\eta$ is D-open, $X\times_\phi X=\coprod_{\eta\in\Lambda}\Gamma_\eta$, and $\phi$ is a Q-chart.
\item For every $h\in\Psi(\phi)$ and $x\in\dom h$ there are $\eta\in\Lambda$ and a D-open $U\ni x$ with $h|_U=\eta|_U$.
\item Let $X$ be elastic and $F,F'\in\Fam$.
  A map $f\colon FX\to F'X$ over $X$ is $\Psi(\phi)$-invariant if and only if $f\circ F\eta=F'\eta\circ f$ for all $\eta\in\Lambda$; a map $f\colon FX\to Z$ is $\Psi(\phi)$-invariant if and only if $f\circ F\eta=f$ for all $\eta\in\Lambda$.
\end{enumerate}
\end{lem}

\begin{proof}
(i) If $P$ is a plot of $X\times_\phi X$, then $\divi\circ P$ is a plot of $\Lambda$, hence locally constant, so $P^{-1}(\Gamma_\eta)$ is open; thus $\Gamma_\eta$ is D-open.
The decomposition is disjoint by freeness.
The diagonal is $\Gamma_e$, which gives (Q2).
The projection $\pr_2$ restricts to a bijection $\Gamma_\eta\to X$ with smooth inverse $x\mapsto(\eta x,x)$, which gives (Q3), and (Q1) holds because quotient maps are subductions.
(ii) The map $u\mapsto\divi(h(u),u)$ is smooth on $\dom h$, hence locally constant as in (i), and $h(u)=\divi(h(u),u)\,u$.
(iii) Each $\eta\in\Lambda$ is a diffeomorphism of $X$ with $\phi\eta=\phi$, i.e.\ an element of $\Psi(\phi)$ with domain $X$; this gives one direction.
Conversely, by (ii) $\dom h$ is covered by D-open sets $U$ on which $h=\eta_U$ for some $\eta_U\in\Lambda$, and by locality $Fh=F\eta_U$ on $FU$.
\end{proof}

\begin{ex}\label{ex:principal}
(a) Let $G$ be a diffeological group and $N\subseteq G$ a subgroup whose subset diffeology is discrete, acting by $\eta\cdot a:=a\eta^{-1}$.
The action is free, and $\divi(b,a)=b^{-1}a$ is smooth, because it is the restriction of a smooth map $G\times G\to G$ with values in $N$.
Hence the action is principal and $G\to G/N$ is a Q-chart; this is the principal fibration of \cite[art.~8.15]{PIZ13} and recovers \cite[Prop.~3.17]{Miy}.
Normality of $N$ is not needed.

(b) Let a countable group $\Lambda$ act freely on $X$.
By \cite[Prop.~3.16]{Miy}, $\phi$ is a Q-chart.
Since $\Gamma_\eta$ is the image of the diagonal under the diffeomorphism $(x,x')\mapsto(\eta x,x')$ of $X\times_\phi X$, it is D-open by (Q2).
Hence $\divi$ is locally constant, and therefore smooth as a map to the diffeologically discrete group $\Lambda$; the action is principal.
\end{ex}

\begin{cor}\label{cor:discrete-quotient}
Let $G$ be an elastic group \textup{(}a diffeological group whose underlying space is elastic\textup{)} and $N\trianglelefteq G$ a normal subgroup whose subset diffeology is discrete, and let $\phi\colon G\to G/N$.
The germs of elements of $\Psi(\phi)$ are precisely the germs of the left translations $L_\eta$, $\eta\in N$.
Consequently every left-invariant vertical connection, horizontal connection and Riemannian metric on $G$ is $\Psi(\phi)$-invariant and hence descends uniquely to $G/N$ by Theorem~\ref{thm:transport}.
If $G$ is connected in the D-topology, $N$ is central.
\end{cor}

\begin{proof}
By Example~\ref{ex:principal}(a) and Lemma~\ref{lem:principal-qchart}(ii), the germs of elements of $\Psi(\phi)$ are germs of right translations $R_\eta(a)=a\eta$, $\eta\in N$.
Since $N$ is normal, each $L_\eta$, $\eta\in N$, satisfies $\phi(\eta a)=\eta aN=aN=\phi(a)$, so $L_\eta\in\Psi(\phi)$.
Conversely, $R_\eta(a)=(a\eta a^{-1})\,a$, and $a\mapsto a\eta a^{-1}$ is a smooth map from $G$ to the discrete space $N$, hence locally constant; so $R_\eta$ coincides near each point with some $L_{\eta'}$.
A left-invariant structure commutes with the tangent maps of all $L_\eta$, hence, by locality of $T$, with those of all elements of $\Psi(\phi)$.
If $G$ is D-connected, $a\mapsto a^{-1}\eta a$ is constant, equal to its value $\eta$ at $a=e$.
\end{proof}

\subsection{Lifting curves}\label{subsec:lifting}

\begin{lem}[Path lifting]\label{lem:global-lift}
Let the action be principal.
For every path $\gamma\colon\R\to Y$ and every $x\in\phi^{-1}(\gamma(0))$ there is a unique path $\tilde\gamma\colon\R\to X$ with $\phi\circ\tilde\gamma=\gamma$ and $\tilde\gamma(0)=x$.
Any two lifts of $\gamma$ differ by the action of a unique $\eta\in\Lambda$.
\end{lem}

When $Y$ is connected, the existence and uniqueness of $\tilde\gamma$ is the case $n=1$ of the lifting of global plots along coverings \cite[art.~8.23]{PIZ13}.
We include a short direct proof, which uses only the division map, does not require $Y$ to be connected, and applies verbatim to the maps $T^n\phi$ of Lemma~\ref{lem:lift-vf}.

\begin{proof}
Let $\tilde\gamma_1,\tilde\gamma_2\colon J\to X$ be lifts of $\gamma|_J$ over an open interval $J$.
Then $t\mapsto\divi(\tilde\gamma_2(t),\tilde\gamma_1(t))$ is a plot of $\Lambda$ on the connected set $J$, hence constant, say $\eta$, and $\tilde\gamma_2=\eta\tilde\gamma_1$ on $J$; if the lifts agree at one point, then $\eta=e$ by freeness.
This proves uniqueness and the last assertion.

For existence, let $I$ be the union of all open intervals $J\ni0$ over which $\gamma$ has a lift with value $x$ at $0$.
By uniqueness these local lifts agree on overlaps, so they glue to a lift $\tilde\gamma\colon I\to X$; moreover $I\neq\emptyset$: since $\phi$ is a subduction, $\gamma$ has a lift $\sigma$ on some open interval around $0$, and $\eta\sigma$ with $\eta:=\divi(x,\sigma(0))$ has value $x$ at $0$.
Suppose $b:=\sup I<\infty$.
There is a lift $\sigma$ of $\gamma$ over an open interval $J\ni b$, and $J\cap I$ is a nonempty open interval.
By the first part, $\tilde\gamma=\eta\sigma$ on $J\cap I$ for some $\eta$, so $\tilde\gamma$ and $\eta\sigma$ glue to a lift over $I\cup J\supsetneq I$, a contradiction.
The same argument rules out a finite lower endpoint, so $I=\R$.
\end{proof}

\begin{lem}\label{lem:lift-vf}
Let the action be principal, $X$ elastic, and $F\in\Fam$.
\begin{enumerate}[label=\textup{(\roman*)},leftmargin=*]
\item The action $\eta\mapsto F\eta$ of $\Lambda$ on $FX$ is principal, and $F\phi\colon FX\to FY$ is its quotient map.
\item Let $n\ge0$ and let $V$ be a vector field on $T^nY$.
  There is a unique vector field $\tilde V$ on $T^nX$ with $T^{n+1}\phi\circ\tilde V=V\circ T^n\phi$; it is $\Lambda$-invariant.
\item In the situation of \textup{(ii)}, let $c\colon\R\to T^nY$ be an integral curve of $V$ and $\xi\in(T^n\phi)^{-1}(c(0))$.
  The lift $\tilde c$ of $c$ through $\xi$ along $T^n\phi$ \textup{(}which exists by \textup{(i)} and Lemma~\textup{\ref{lem:global-lift})} is an integral curve of $\tilde V$.
  Conversely, $T^n\phi$ maps integral curves of $\tilde V$ to integral curves of $V$.
\end{enumerate}
\end{lem}

\begin{proof}
(i) The action $\Lambda\times FX\to FX$ is smooth, since every plot of $\Lambda$ is locally constant and each $F\eta$ is smooth.
If $F\eta(\xi)=\xi$, then $\eta\,\pi^F_X(\xi)=\pi^F_X(\xi)$ and $\eta=e$.
The fibres of $F\phi$ are the orbits: clearly $F\phi\circ F\eta=F\phi$, and if $F\phi(\xi')=F\phi(\xi)$, Lemma~\ref{lem:miy-toolkit}(iii) gives $h\in\Psi(\phi)$ with $Fh(\xi')=\xi$, and $Fh=F\eta$ near $\xi'$ for some $\eta$ by Lemma~\ref{lem:principal-qchart}(ii) and locality.
The division map is $(\xi,\xi')\mapsto\divi(\pi^F_X\xi,\pi^F_X\xi')$, which is smooth.
Since $F\phi$ is a subduction, it is the quotient map.

(ii) Every vector field $V\colon T^nM\to T^{n+1}M$ is automatically a map over $M$, since $\pi^{T^{n+1}}_M\circ V=\pi^{T^n}_M\circ\pi_{T^nM}\circ V=\pi^{T^n}_M$.
By Proposition~\ref{prop:descent}(ii) there is a unique $\Psi(\phi)$-invariant $\tilde V$ over $X$ with $T^{n+1}\phi\,\tilde V= V\,T^n\phi$.
It is a section: $\pi_{T^nX}\tilde V$ and $\id$ are maps over $X$ and $T^n\phi\,\pi_{T^nX}\tilde V=\pi_{T^nY}\,V\,T^n\phi=T^n\phi$, so Corollary~\ref{cor:detect} applies.
Uniqueness among all vector fields follows from Corollary~\ref{cor:detect} as well, and invariance from Lemma~\ref{lem:principal-qchart}(iii).

(iii) The maps $T(\tilde c)\circ\partial$ and $\tilde V\circ\tilde c$ from $\R$ to $T^{n+1}X$ have the same image under the base projection $\pi^{T^{n+1}}_X$, namely $\pi^{T^n}_X\circ\tilde c$, and
\[
  T^{n+1}\phi\,T(\tilde c)\,\partial=T(T^n\phi\,\tilde c)\,\partial=T(c)\,\partial
  =V\,c=V\,T^n\phi\,\tilde c=T^{n+1}\phi\,\tilde V\,\tilde c,
\]
so they coincide by Corollary~\ref{cor:detect}.
The converse is the same computation read backwards.
\end{proof}

\subsection{Geodesics and curve objects}\label{subsec:geodesics-global}

From now on $X$ is elastic, the action is principal, and $K$, $H$, $g$ are $\Lambda$-invariant (equivalently, $\Psi(\phi)$-invariant by Lemma~\ref{lem:principal-qchart}(iii)); $K_Y$, $H_Y$, $g_Y$ denote the structures on $Y$ given by Theorem~\ref{thm:transport}.

\begin{prop}\label{prop:geodesic-converse}
Let $\gamma_Y$ be a geodesic of $K_Y$ and $x\in\phi^{-1}(\gamma_Y(0))$.
The lift $\tilde\gamma$ of $\gamma_Y$ through $x$ is a geodesic of $K$ with $\dot{\tilde\gamma}(0)=(T_x\phi)^{-1}(\dot\gamma_Y(0))$.
Consequently $\gamma\mapsto\phi\circ\gamma$ maps the set of geodesics of $K$ onto the set of geodesics of $K_Y$, and its fibres are precisely the orbits of the free action of $\Lambda$ on the geodesics of $K$ given by $(\eta\cdot\gamma)(t):=\eta\gamma(t)$.
\end{prop}

\begin{proof}
The lift exists by Lemma~\ref{lem:global-lift} and is a geodesic by Theorem~\ref{thm:transport}(5).
Since $T\phi\circ\dot{\tilde\gamma}=\dot\gamma_Y$ and $T_x\phi$ is bijective, the initial velocity is as stated.
For $\eta\in\Lambda$ and a geodesic $\gamma$ of $K$, $\phi\circ\eta\circ\gamma= \phi\circ\gamma$ is a geodesic of $K_Y$, so $\eta\circ\gamma$ is a geodesic of $K$ by Theorem~\ref{thm:transport}(5).
Two geodesics with the same image under $\phi$ are lifts of the same path and hence differ by a unique $\eta$ (Lemma~\ref{lem:global-lift}); $\eta\circ\gamma=\gamma$ forces $\eta=e$ by freeness.
\end{proof}

\begin{thm}\label{thm:elast-curve-quotient}
If $X\in\Elastc$, then $Y\in\Elastc$.
\end{thm}

\begin{proof}
Fix $n\ge0$, a vector field $V$ on $T^nY$, a map $s\colon Z\to T^nY$, and two solutions $\Phi,\Phi'\colon\R\times Z\to T^nY$ of $(T^nY,V,s)$.
For $z\in Z$ the slices $\Phi_z:=\Phi(-,z)$ and $\Phi'_z$ are integral curves of $V$ with initial value $s(z)$: indeed $T(\Phi_z)\circ\partial=T(\Phi)\circ(\partial\times0_Z)\circ \iota_z=V\circ\Phi_z$, where $\iota_z(t)=(t,z)$ and $T(\iota_z)\circ\partial=(\partial\times0_Z)\circ\iota_z$ by naturality of the zero section (cf.\ \cite[Rem.~6.20]{Shio}).
Choose $\xi\in(T^n\phi)^{-1}(s(z))$.
By Lemma~\ref{lem:lift-vf}, the lifts of $\Phi_z$ and $\Phi'_z$ through $\xi$ are integral curves of $\tilde V$ with the same initial value, so they coincide because $X\in\Elastc$.
Applying $T^n\phi$ gives $\Phi_z=\Phi'_z$.
As $z$ was arbitrary, $\Phi=\Phi'$.
\end{proof}

\begin{prop}\label{prop:flow-descends}
Let $X\in\Elastc$, let $(K,H)$ be a $\Lambda$-invariant full connection whose geodesic spray $S$ is complete, and let $\Phi_X$ be its geodesic flow \textup{(}Convention~\ref{conv:flow}\textup{)}.
Then there exists a unique smooth map $\Phi_Y\colon\R\times TY\to TY$ such that
\[
  \Phi_Y\circ(\id_\R\times T\phi)=T\phi\circ\Phi_X .
\]
Moreover, $\Phi_Y$ is a geodesic flow of $(K_Y,H_Y)$; in particular, the geodesic spray $S_Y$ of $(K_Y,H_Y)$ is complete.
\end{prop}

\begin{proof}
Since $\langle\id,\id\rangle\circ T\phi=T_2\phi\circ\langle\id,\id\rangle$, \eqref{eq:H-rel} gives $S_Y\circ T\phi=T^2\phi\circ S$, and $S$ is $\Lambda$-invariant: $S\circ T\eta=T^2\eta\circ S$.
For $\eta\in\Lambda$ both $\Phi_X\circ(\id_\R\times T\eta)$ and $T\eta\circ\Phi_X$ solve the parametrized system $(TX,S,T\eta)$ with parameter object $TX$; for the first this uses $T(T\eta)\circ0_{TX}=0_{TX}\circ T\eta$, for the second the invariance of $S$.
As $X\in\Elastc$ they coincide.
By Lemma~\ref{lem:lift-vf}(i) the fibres of $\id_\R\times T\phi$ are the sets $\{t\}\times(\Lambda\text{-orbit})$, so $T\phi\circ\Phi_X$ is constant on them, and since $\id_\R\times T\phi$ is a subduction, Fact~\ref{fact:subduction-coeq} yields a unique $\Phi_Y$ as stated.
It solves $(TY,S_Y,\id_{TY})$: first $\Phi_Y(0,T\phi(v))=T\phi(\Phi_X(0,v))=T\phi(v)$; second, using $T(\id_\R\times T\phi)\circ(\partial\times0_{TX})=(\partial\times0_{TY})\circ (\id_\R\times T\phi)$,
\[
  T(\Phi_Y)\,(\partial\times0_{TY})\,(\id_\R\times T\phi)
  =T(T\phi\,\Phi_X)\,(\partial\times0_{TX})
  =T^2\phi\,S\,\Phi_X=S_Y\,T\phi\,\Phi_X
  =S_Y\,\Phi_Y\,(\id_\R\times T\phi),
\]
and $\id_\R\times T\phi$ is surjective.
\end{proof}

\begin{cor}\label{cor:uniqueness}
Let $X\in\Elastc$, and let $(K,H)$ be a $\Lambda$-invariant full connection with complete geodesic spray.
Then $Y\in\Elastc$, the geodesic spray of $(K_Y,H_Y)$ is complete, and every $w\in TY$ is the initial velocity of exactly one geodesic of $K_Y$, namely $\phi\circ\gamma$, where $\gamma$ is the geodesic of $K$ with initial velocity $(T_x\phi)^{-1}(w)$ for any $x\in\phi^{-1}(\pi_Y(w))$.
\end{cor}

\begin{proof}
The first two assertions are Theorem~\ref{thm:elast-curve-quotient} and Proposition~\ref{prop:flow-descends}.
Let $v:=(T_x\phi)^{-1}(w)$.
By Lemma~\ref{lem:geod-spray}(ii), $\gamma:=\pi_X\circ\Phi_X(-,v)$ is a geodesic of $K$ with $\dot\gamma(0)=v$, so $\phi\circ\gamma$ is a geodesic of $K_Y$ with initial velocity $T\phi(v)=w$ by Theorem~\ref{thm:transport}(5).
Uniqueness is Lemma~\ref{lem:geod-spray}(iii) applied to $Y\in\Elastc$; in particular the geodesic $\phi\circ\gamma$ does not depend on the choice of $x$.
\end{proof}

\begin{rem}\label{rem:variational}
Under the hypotheses of Corollary~\ref{cor:uniqueness}, if moreover $K$ is a Levi-Civita connection of a $\Lambda$-invariant metric $g$, then $Y\in\Elastc$ carries the full connection $(K_Y,H_Y)$ whose vertical part is the Levi-Civita connection of $g_Y$ (Theorem~\ref{thm:transport}(4)) and whose geodesic spray is complete.
Thus all standing hypotheses of the variational theory of \cite[\S\S6.4--6.5]{Shio} are satisfied on $Y$; in particular, on $Y$ the critical paths of the energy functional are precisely the geodesics \cite[Props.~6.30, 6.33]{Shio}, and length-minimizing regular paths are geodesics after reparametrization \cite[Prop.~6.37]{Shio}.
\end{rem}

% =====================================================================
\section{Application: irrational tori}\label{sec:irr}
% =====================================================================

Fix $\alpha\in\R\setminus\Q$ and put
\[
  \Lambda_\alpha:=\Z+\alpha\Z\subseteq\R,\qquad
  \Tor_\alpha:=\R/\Lambda_\alpha,\qquad
  q\colon\R\to\Tor_\alpha,\quad [x]:=q(x),
\]
where $\Lambda_\alpha$ acts on $\R$ by the translations $L_\eta(x)=x+\eta$.
In this section we apply the preceding descent results to $\Tor_\alpha$.
Although its D-topology is indiscrete, $\Tor_\alpha$ carries a rigid Riemannian geometry: its metrics and connections can be classified completely, each metric has a unique Levi-Civita connection with complete geodesic flow, and yet the induced pseudo-distance vanishes identically.
We write $t$ for the coordinate of $\R$, $\partial$ for the unit vector field and $dt$ for the standard $1$-form.
For the descended objects on $\Tor_\alpha$ we use the distinct symbols
\[
  \bar\partial([x])=([x],1),\qquad \vartheta([x],v)=v,
\]
so that $Tq\circ\partial=\bar\partial\circ q$ and $q^*\vartheta=dt$.

\subsection{The irrational torus as a quotient}\label{subsec:irr-quotient}

\begin{lem}\label{lem:irr-basic}
\begin{enumerate}[label=\textup{(\roman*)},leftmargin=*]
\item $\Lambda_\alpha$ is dense in $\R$, and its subset diffeology is discrete.
\item The action of $\Lambda_\alpha$ on $\R$ is principal, $q$ is a Q-chart, and a structure on $\R$ is $\Psi(q)$-invariant if and only if it is invariant under all translations $L_\eta$, $\eta\in\Lambda_\alpha$.
\item $\Tor_\alpha$ is an elastic group and $\Tor_\alpha\in\Elastc$.
\item The D-topology of $\Tor_\alpha$ is indiscrete.
\item For $n\ge0$ and $k\ge1$ there are diffeomorphisms $T^n\Tor_\alpha\cong\Tor_\alpha\times\R^{2^n-1}$ and $T_k\Tor_\alpha\cong\Tor_\alpha\times\R^k$ under which $T^nq$ and $T_kq$ become $q\times\id$.
\end{enumerate}
\end{lem}

\begin{proof}
(i) $\Lambda_\alpha$ is not cyclic because $\alpha$ is irrational, and every non-cyclic subgroup of $\R$ is dense.
A plot $P\colon U\to\Lambda_\alpha$ is, by definition of the subset diffeology, smooth as a map $U\to\R$.
On each connected component of $U$ its image is connected and countable, hence a singleton.
Since the components of $U$ are open, $P$ is locally constant.
(ii) The division map $(x',x)\mapsto x'-x$ is smooth, so this is Example~\ref{ex:principal}(a) together with Lemma~\ref{lem:principal-qchart}(iii).
(iii) $\Tor_\alpha$ is elastic by \cite[Thm.~3.24]{Miy} (see also \cite[Cor.~3.25]{Miy}); $\R\in\Elastc$ \cite[Rem.~6.20]{Shio}, so $\Tor_\alpha\in\Elastc$ by Theorem~\ref{thm:elast-curve-quotient}.
(iv) Let $V\subseteq\Tor_\alpha$ be D-open and nonempty.
Then $q^{-1}(V)$ is a nonempty open $\Lambda_\alpha$-invariant subset of $\R$; it contains an open interval $I$.
For $x\in\R$ the open set $x-I$ contains some $\eta\in\Lambda_\alpha$ by density, so $x\in I+\eta\subseteq q^{-1}(V)$.
Hence $V=\Tor_\alpha$.
(v) We have $T^n\R=\R\times\R^{2^n-1}$ with $T^nL_\eta=L_\eta\times\id$.
By Lemma~\ref{lem:lift-vf}(i), $T^nq$ is a subduction whose fibres are the orbits of $\Lambda_\alpha$ acting on the first factor.
The map $q\times\id$ is also a subduction (a product of subductions) with the same fibres, so the induced bijection $T^n\Tor_\alpha\to\Tor_\alpha\times\R^{2^n-1}$ is a diffeomorphism by Fact~\ref{fact:subduction-coeq}.
The case of $T_k$ is identical.
\end{proof}

The following elementary consequence of density will be used repeatedly to turn $\Lambda_\alpha$-invariance into constancy.

\begin{lem}\label{lem:dense-const}
Let $Z$ be a Hausdorff space and $f\colon\R\to Z$ a continuous map with $f(x+\eta)=f(x)$ for all $x\in\R$ and $\eta\in\Lambda_\alpha$.
Then $f$ is constant.
\end{lem}

\begin{proof}
For $x\in\R$ choose $\eta_n\in\Lambda_\alpha$ with $\eta_n\to x$; then $f(x)=\lim f(\eta_n)=f(0)$.
\end{proof}

We use the identifications of Lemma~\ref{lem:irr-basic}(v) and write points of $T^2\Tor_\alpha$ as $([x],v_0,v_1,v_{01})$, and so on.
Since the structure maps are natural and $T^nq=q\times\id$, they are given on $\Tor_\alpha$ by the same formulas as on $\R$ (Remark~\ref{rem:coordinates}).

\begin{rem}\label{rem:irr-context}
The group $\Tor_\alpha$ is also the quotient of the $2$-torus by the one-parameter subgroup of slope $\alpha$ \cite[Prop.~45]{PIZ25}; this quotient map is a fibration in diffeology but not in the category of manifolds \cite[art.~8.15, Note]{PIZ13}.
\end{rem}

\subsection{Invariant geometry on the line}\label{subsec:line}

\begin{lem}\label{lem:euclid-vconn}
Let $U\subseteq\R^n$ be open.
Every vertical connection on $U$ is of the form
\[
  K_\Gamma(u,v_0,v_1,v_{01})=\bigl(u,\;v_{01}+\Gamma(u)(v_0,v_1)\bigr)
\]
for a unique smooth family $\Gamma$ of bilinear maps $\R^n\times\R^n\to\R^n$, and every such $K_\Gamma$ is a vertical connection \cite[Ex.~3.5]{Shio}.
The connection $K_\Gamma$ is torsion-free if and only if every $\Gamma(u)$ is symmetric.
It is always effective; its partner is $H_\Gamma\bigl((u,v_0),(u,v_1)\bigr)=\bigl(u,v_1,v_0,-\Gamma(u)(v_1,v_0)\bigr)$ and its geodesic spray is $S_\Gamma(u,v)=\bigl(u,v,v,-\Gamma(u)(v,v)\bigr)$.
\end{lem}

\begin{proof}
By axiom (b), $K(u,v_0,v_1,v_{01})=(u,k(u,v_0,v_1,v_{01}))$ with $k$ smooth.
By axioms (e), (f) and Remark~\ref{rem:coordinates}, for all $t\in\R$,
\[
  k(u,v_0,tv_1,tv_{01})=t\,k(u,v_0,v_1,v_{01}),\qquad
  k(u,tv_0,v_1,tv_{01})=t\,k(u,v_0,v_1,v_{01}).
\]
A smooth function $f$ with $f(tw)=tf(w)$ for all $t$ is linear, since $f(w)=\frac{d}{dt}\big|_{t=0}f(tw)=Df(0)w$.
Hence $k$ is linear in $(v_1,v_{01})$ for fixed $(u,v_0)$, and linear in $(v_0,v_{01})$ for fixed $(u,v_1)$.
Thus
\[
  k(u,v_0,v_1,v_{01})=k(u,v_0,v_1,0)+k(u,v_0,0,v_{01}),
\]
where the first term is bilinear in $(v_0,v_1)$, say $\Gamma(u)(v_0,v_1)$, and the second is linear in $(v_0,v_{01})$ and vanishes for $v_{01}=0$, hence equals $B(u)v_{01}$.
Axiom (a) gives $K(u,0,0,v)=(u,B(u)v)=(u,v)$, so $B(u)=\id$.
The statement on torsion follows from $\tau(u,v_0,v_1,v_{01})=(u,v_1,v_0,v_{01})$.
The map
\[
  \Theta(u,v_0,v_1,v_{01})=\bigl((u,v_0),(u,v_1),(u,v_{01}+\Gamma(u)(v_0,v_1))\bigr)
\]
has the smooth inverse obtained by subtracting $\Gamma(u)(v_0,v_1)$.
By axioms (a), (b) of Definition~\ref{def:hconn}, the partner has the form $H((u,v_0),(u,v_1))=(u,v_1,v_0,w)$, and (F1) forces $w=-\Gamma(u)(v_1,v_0)$.
Finally $S_\Gamma(v)=H_\Gamma(v,v)$.
\end{proof}

For $n=1$ every $\Gamma$ is of the form $\Gamma(u)(v_0,v_1)=c(u)v_0v_1$ with $c\in C^\infty(U)$, and we write $K_c:=K_\Gamma$, $H_c:=H_\Gamma$ and $S_c:=S_\Gamma$.
In particular, $K_0(u,v_0,v_1,v_{01})=(u,v_{01})$ is the trivial connection.

\begin{lem}\label{lem:line}
Let $\rho>0$ and $c$ be constants.
\begin{enumerate}[label=\textup{(\roman*)},leftmargin=*]
\item $K_c$ is flat.
\item $K_c$ is a Levi-Civita connection of $\rho\,dt^2$ if and only if $c=0$.
\item The geodesic spray $S_c$ is complete if and only if $c=0$.
  The geodesic flow of $(K_0,H_0)$ is $\Phi(t,(x,v))=(x+tv,v)$, and the geodesics of $K_0$ are the affine maps $t\mapsto x+tv$.
\end{enumerate}
\end{lem}

\begin{proof}
(i) Write points of $T^3\R$ as $(x,v_0,v_1,v_{01};v_2,v_{02},v_{12},v_{012})$, the last four coordinates being the derivative in the new direction $2$ as in \cite[Ex.~5.2]{Shio}.
Then
\[
  T(K_c)=\bigl(x,\;v_{01}+cv_0v_1,\;v_2,\;v_{012}+c(v_{02}v_1+v_0v_{12})\bigr),
\]
\[
  K_c\,T(K_c)=\bigl(x,\;v_{012}+c\,v_{02}v_1+c\,v_0v_{12}+c\,v_{01}v_2+c^2v_0v_1v_2\bigr).
\]
The flip $\tau_{T\R}$ exchanges the directions $1$ and $2$, i.e.\ $v_1\leftrightarrow v_2$ and $v_{01}\leftrightarrow v_{02}$, fixing $v_0$, $v_{12}$, $v_{012}$ \cite[Ex.~5.2]{Shio}.
The expression above is invariant under this exchange, so $C_{K_c}=0$.

(ii) Let $g=\rho\,dt^2$.
We have $T(\partial)(t,w)=(t,1,w,0)$, so $\nabla^{K_c}_\partial\partial=K_c(t,1,1,0)=c\,\partial$.
If $K_c$ is a Levi-Civita connection, \cite[Prop.~4.12]{Shio} gives $0=\partial\bigl(g(\partial,\partial)\bigr)=2g(\nabla^{K_c}_\partial\partial,\partial) =2c\rho$, so $c=0$.
Conversely, $K_0$ is torsion-free.
Write points of $T(T_2\R)$ as $(x,a,b;\dot x,\dot a,\dot b)$; under \eqref{eq:chi} they correspond to the pairs $A=(x,a,\dot x,\dot a)$, $B=(x,b,\dot x,\dot b)$ in $T^2\R$.
Since $g(x,a,b)=\rho ab$, we get $d(g)=\rho(\dot ab+a\dot b)$, while $g(K_0A,\pi_{T\R}B)+g(\pi_{T\R}A,K_0B)=\rho\dot a\,b+\rho a\,\dot b$.
Hence \eqref{eq:LC} holds.

(iii) By Lemma~\ref{lem:euclid-vconn}, $S_c(x,v)=(x,v,v,-cv^2)$, so the integral curves of $S_c$ solve $\dot x=v$, $\dot v=-cv^2$.
For $c\neq0$ and $v(0)\neq0$ the solution $v(t)=v(0)/(1+cv(0)t)$ does not extend to all of $\R$, so no geodesic flow exists.
For $c=0$ the integral curves are $t\mapsto(x+tv,v)$.
The last assertion follows from Lemma~\ref{lem:geod-spray}.
\end{proof}

\subsection{Classification of Riemannian structures}\label{subsec:irr-geometry}

\begin{thm}\label{thm:irr}
\begin{enumerate}[label=\textup{(\roman*)},leftmargin=*]
\item $C^\infty(\Tor_\alpha,\R)=\R$.
  The vector fields on $\Tor_\alpha$ are precisely $a\,\bar\partial$, $a\in\R$, and $\Omega^1(\Tor_\alpha)=\R\,\vartheta$; the form $\vartheta$ is closed but not exact.
\item The Riemannian metrics on $\Tor_\alpha$ are exactly $g_\rho=\rho\,\vartheta^2$, $g_\rho([x],v_0,v_1)=\rho\,v_0v_1$, with $\rho>0$ a constant; every $T_xq\colon(T_x\R,\rho\,dt^2)\to(T_{[x]}\Tor_\alpha,g_\rho)$ is a linear isometry.
  In particular the set of Riemannian metrics on $\Tor_\alpha$ is parametrized by $\rho\in\R_{>0}$.
\item The vertical connections on $\Tor_\alpha$ are exactly
  \[
    K_c([x],v_0,v_1,v_{01})=([x],\,v_{01}+c\,v_0v_1),\qquad c\in\R .
  \]
All of them are torsion-free, effective and flat; the partner of $K_c$ is
  \[
        H_c\bigl(([x],v_0),([x],v_1)\bigr)=([x],v_1,v_0,-c\,v_0v_1).
  \]
\item For every $\rho>0$, $K_0$ is the unique Levi-Civita connection of $g_\rho$.
  The geodesic spray of $(K_0,H_0)$ is complete, with geodesic flow $\Phi(t,([x],v))=([x+tv],v)$, and the geodesics of $K_0$ are exactly the paths $t\mapsto[x+tv]$.
  For $c\neq0$ the geodesic spray of $(K_c,H_c)$ is not complete.
\end{enumerate}
\end{thm}

\begin{proof}
By Theorem~\ref{thm:transport}, Proposition~\ref{prop:descent} and Lemma~\ref{lem:irr-basic}(ii), each kind of structure on $\Tor_\alpha$ corresponds bijectively to the $\Lambda_\alpha$-invariant structures of the same kind on $\R$, compatibly with all properties listed there.
By Lemma~\ref{lem:dense-const}, every $\Lambda_\alpha$-invariant continuous coefficient on $\R$ is constant; this is used in each step.

(i) A smooth function on $\Tor_\alpha$ is $f_Y$ for a $\Lambda_\alpha$-invariant $f\in C^\infty(\R)$, hence constant.
By Proposition~\ref{prop:descent}(ii) with $F=T^0$, $F'=T$ (cf.\ Lemma~\ref{lem:lift-vf}(ii) with $n=0$), vector fields on $\Tor_\alpha$ correspond to $\Lambda_\alpha$-invariant vector fields on $\R$.
Every vector field on $\R$ is $a(t)\,\partial$, and invariance gives $a(t+\eta)=a(t)$ for $\eta\in\Lambda_\alpha$; hence $a$ is constant and the vector field on $\Tor_\alpha$ is $a\,\bar\partial$.
By \cite[Lem.~2.15]{Shio}, $\Omega^1(M)$ is the set of smooth maps $TM\to\R$ that are linear on the fibres.
Such maps on $T\Tor_\alpha$ correspond to $\Lambda_\alpha$-invariant maps on $T\R$ (Proposition~\ref{prop:descent}(i)), and fibrewise linearity is preserved because each $T_xq$ is a linear bijection.
They lift to the $\Lambda_\alpha$-invariant forms $b\,dt$ on $\R$, so $b$ is constant and they are the forms $b\,\vartheta$.
The form $\vartheta$ is closed, since $q^*d\vartheta=d(dt)=0$ and $q^*$ is injective on forms ($q$ being a subduction); it is not exact because the only smooth functions are constants.

(ii) By Theorem~\ref{thm:transport}(3), metrics on $\Tor_\alpha$ correspond to $\Lambda_\alpha$-invariant metrics $\rho(t)\,dt^2$ on $\R$, and $\rho$ is constant.

(iii) By Theorem~\ref{thm:transport}(1) and Lemma~\ref{lem:euclid-vconn}, vertical connections on $\Tor_\alpha$ correspond to the connections $K_c$ on $\R$ with $c\in C^\infty(\R)$ and $c(t+\eta)=c(t)$ for $\eta\in\Lambda_\alpha$, i.e.\ to constant $c$; the formula follows from $K_c\circ T^2q=Tq\circ K_c^{\R}$ and $T^nq=q\times\id$.
Torsion-freeness, effectivity and flatness transfer from $\R$ by Theorem~\ref{thm:transport}(1) and Lemmas~\ref{lem:euclid-vconn}, \ref{lem:line}(i), and the partner is the descent of $H_c$ by Theorem~\ref{thm:transport}(2).

(iv) By Theorem~\ref{thm:transport}(4), the Levi-Civita connections of $g_\rho$ correspond to the $\Lambda_\alpha$-invariant Levi-Civita connections of $\rho\,dt^2$ on $\R$, which by (iii) and Lemma~\ref{lem:line}(ii) reduce to $K_0$.
Since $\R\in\Elastc$ and the spray of $(K_0,H_0)$ on $\R$ is complete, Proposition~\ref{prop:flow-descends} and Corollary~\ref{cor:uniqueness} give the flow and the geodesics.
Suppose that the spray $S$ of $(K_c,H_c)$ on $\Tor_\alpha$ is complete for some $c\neq0$, and let $v\neq0$.
Then $t\mapsto\Phi(t,([0],v))$ is an integral curve of $S$ defined on $\R$.
The spray $S^{\R}_c$ on $\R$ satisfies $T^2q\circ S_c^{\R}=S\circ Tq$ (proof of Proposition~\ref{prop:flow-descends}), so it is the lift $\tilde S$ of Lemma~\ref{lem:lift-vf}(ii) for $n=1$.
By Lemma~\ref{lem:lift-vf}(iii) we obtain an integral curve of $S_c^{\R}$ through $(0,v)$ defined on all of $\R$, contradicting Lemma~\ref{lem:line}(iii).
\end{proof}

\subsection{Metric and topological consequences}\label{subsec:irr-distance}

For a Riemannian metric $g$ on an elastic space $M$ and a path $\gamma$, put $L(\gamma):=\int_0^1g(\dot\gamma(t),\dot\gamma(t))^{1/2}\,dt$; the integrand is continuous.
For regular paths this is the length of \cite[Def.~6.34]{Shio}.
Let
\[
  d_g(x,y):=\inf\{L(\gamma)\mid\gamma\in\Path(M;x,y)\}\in[0,\infty],
\]
where $\Path(M;x,y)$ is the set of paths with $\gamma(0)=x$ and $\gamma(1)=y$.
Recall that paths are defined on all of $\R$; only the integral is taken over $[0,1]$.
This agrees with \cite[\S6]{Shio} and \cite[(3.6)]{KSS}.

\begin{prop}\label{prop:d-zero}
For every $\rho>0$, the pseudo-distance $d_{g_\rho}$ on $\Tor_\alpha$ vanishes identically, and for $[x]\neq[y]$ no path from $[x]$ to $[y]$ realizes the infimum.
The same holds if the infimum is taken over regular paths only.
\end{prop}

\begin{proof}
Let $\gamma\in\Path(\Tor_\alpha;[x],[y])$ and let $\tilde\gamma$ be its lift with $\tilde\gamma(0)=x$ (Lemma~\ref{lem:global-lift}); then $\tilde\gamma(1)=y+\eta$ for some $\eta\in\Lambda_\alpha$.
Since each $T_xq$ is an isometry, $L(\gamma)=\sqrt\rho\int_0^1|\tilde\gamma'(t)|\,dt\ge\sqrt\rho\, |y+\eta-x|$, with equality for the affine path from $x$ to $y+\eta$.
Every $\eta\in\Lambda_\alpha$ arises in this way, so
\[
  d_{g_\rho}([x],[y])=\inf_{\eta\in\Lambda_\alpha}\sqrt\rho\,|y+\eta-x|=0
\]
by density.
The affine paths with $y+\eta\neq x$ are regular, and the infimum of $|y+\eta-x|$ over $\eta\in\Lambda_\alpha\setminus\{x-y\}$ is still $0$ because $\Lambda_\alpha\setminus\{x-y\}$ is dense; this gives the last assertion.
(For $[x]=[y]$ and $y=x$ this means that there are regular loops at $[x]$ of arbitrarily small length.)
If $L(\gamma)=0$, then $\tilde\gamma'=0$ on $[0,1]$, so $[y]=q(\tilde\gamma(1))=q(\tilde\gamma(0))=[x]$.
\end{proof}

\begin{rem}[Comparison with \cite{KSS}]\label{rem:KSS-distance}
The pseudo-distance $d_g$ is the one of \cite[(3.6)]{KSS}; on elastic spaces the weak Riemannian metrics of \cite{KSS} are the maps $T_2M\to\R$ that are bilinear, symmetric and positive on the fibres \cite[Rem.~4.4]{Shio}.
By Remark~\ref{rem:KSS-definite}, $g_\rho$ is definite in the sense of \cite[Def.~3.7]{KSS} with respect to the generating family $\{q\}$.
Proposition~\ref{prop:d-zero} is consistent with the results of \cite{KSS}:
\begin{enumerate}[label=\textup{(\alph*)},leftmargin=*]
\item \cite[Thm.~3.27]{KSS} asserts that $d_g$ is a distance if $g$ is definite with respect to a generating family which \emph{separates points} in the sense of \cite[Def.~3.25]{KSS}.
  By Proposition~\ref{prop:d-zero}, no generating family of the diffeology of $\Tor_\alpha$ with respect to which $g_\rho$ is definite separates points.
\item \cite[Thm.~3.21]{KSS} asserts that on a connected diffeological space the D-topology is finer than the topology $\mathcal O_d$ defined by $d_g$.
  The space $\Tor_\alpha$ is connected, being the image of $\R$.
  Since the D-topology of $\Tor_\alpha$ is indiscrete (Lemma~\ref{lem:irr-basic}(iv)), every open $d_{g_\rho}$-ball is all of $\Tor_\alpha$; this gives a second proof that $d_{g_\rho}\equiv0$.
\end{enumerate}
In \cite[Ex.~4.3]{KSS} a definite metric on the line with a doubled half-line has pseudo-distance zero between a single pair of distinct points; on $\Tor_\alpha$ it vanishes between all pairs.
\end{rem}

\begin{rem}\label{rem:pi1}
The irrational torus is a standard example showing that a diffeology carries substantially more information than its D-topology.
The D-topology of $\Tor_\alpha$ is indiscrete (Lemma~\ref{lem:irr-basic}(iv)), as is the quotient topology of $\R/\Lambda_\alpha$; nevertheless its diffeological homotopy theory is nontrivial, since $\pi_1(\Tor_\alpha)\cong\Lambda_\alpha\cong\Z^2$ \cite[art.~8.38]{PIZ13}.
As shown in this section, it also carries a nontrivial Riemannian geometry: the metrics $\rho\,dt^2$ on $\R$ descend to the metrics $g_\rho=\rho\,\vartheta^2$ on $\Tor_\alpha$, each of which has a unique Levi-Civita connection, namely $K_0$, with complete geodesic flow, although the induced pseudo-distance vanishes identically.
Thus the diffeology retains both homotopical and differential-geometric information that is invisible to the D-topology.
\end{rem}

\end{document}